\documentclass[10pt,a4paper,reqno]{amsart}

\makeatletter
\@namedef{subjclassname@2020}{%
  \textup{2020} Mathematics Subject Classification}
\makeatother
\usepackage{etex}
\usepackage[left=3cm,right=3cm,top=3cm,bottom=3cm]{geometry}
\usepackage{booktabs}
\usepackage{siunitx}
\usepackage{thmtools, thm-restate}

\usepackage[toc]{appendix}
\usepackage{amsmath}
\usepackage{amssymb}
\usepackage{amsthm}
\usepackage{mathrsfs}
\usepackage{latexsym}
\usepackage{amscd}
\usepackage{enumitem}
\usepackage[all]{xy}
\usepackage{ifthen} 
\usepackage{geometry}
\usepackage{graphicx}
\usepackage{rotating}
\usepackage[dvipsnames]{xcolor}
\usepackage{mathtools}

\usepackage{dsfont}
\usepackage[libertine,scaled=1]{newtxmath}
\usepackage{todonotes}
\usepackage{color}
\usepackage{soul}
\usepackage[only,llbracket,rrbracket]{stmaryrd}

\usepackage[pagebackref]{hyperref}

\usepackage{comment}

\usepackage{float}
\restylefloat{table}
\numberwithin{equation}{section}

\def\cocoa{{\hbox{\rm C\kern-.13em o\kern-.07em C\kern-.13em o\kern-.15em A}}}

\declaretheorem[numberwithin=section]{theorem}
\newtheorem{lemma}[theorem]{Lemma}

\newtheorem{proposition}[theorem]{Proposition}
\newtheorem{corollary}[theorem]{Corollary}

\theoremstyle{definition}
\newtheorem*{ack}{Acknowledgements}
\newtheorem{remark}[theorem]{Remark}
\newtheorem{definition}[theorem]{Definition}
\newtheorem{example}[theorem]{Example}

\newtheorem{question}[theorem]{Question}

\newcommand {\uc}{\mathrm{uc}}
\newcommand {\PGL}{\mathrm{PGL}}
\newcommand {\Pf}{\mathrm{Pf}}

\newcommand {\coker}{\mathrm{coker}}

\newcommand {\ch}{\mathrm{ch}}

\newcommand {\sHom}{\mathcal{H}\kern -0.25ex{\mathit om}}
\newcommand {\sExt}{\mathcal{E}\kern -0.25ex{\mathit xt}}
\newcommand {\sTor}{\mathcal{T}\kern -0.25ex{\mathit or}}
\newcommand {\im}{\mathrm{im}}
\newcommand {\rk}{\mathrm{rk}}

\newcommand {\Ext}{\mathrm{Ext}}
\newcommand {\Hom}{\mathrm{Hom}}
\newcommand {\Hilb}{\mathrm{Hilb}}

\newcommand {\Aut}{\mathrm{Aut}}

\newcommand {\field}{\bk}

\newcommand {\Db}{\mathrm{D}^b}
\newcommand {\Ku}{\mathrm{Ku}}
\newcommand {\At}{\mathrm{At}}
\newcommand {\tr}{\mathrm{Tr}}
\newcommand {\quantum}{k}

\newcommand {\fM}{\mathfrak{M}}
\newcommand {\fU}{\mathfrak{U}}

\newcommand {\cK}{\mathcal{K}}
\newcommand {\cA}{\mathcal{A}}
\newcommand {\cB}{\mathcal{B}}

\newcommand {\cU}{\mathcal{U}}

\newcommand{\cC}{{\mathcal C}}
\newcommand{\cS}{{\mathcal S}}
\newcommand{\cE}{{\mathcal E}}
\newcommand{\cF}{{\mathcal F}}

\newcommand{\cN}{{\mathcal N}}
\newcommand{\cO}{{\mathcal O}}
\newcommand{\cG}{{\mathcal G}}
\newcommand{\cT}{{\mathcal T}}
\newcommand{\cTF}{{\mathcal F_0}}
\newcommand{\cI}{{\mathcal I}}
\newcommand{\fI}{{\mathfrak I}}

\newcommand {\bk}{\mathds{k}}
\newcommand {\bN}{\mathds{N}}
\newcommand {\bZ}{\mathds{Z}}
\newcommand {\bQ}{\mathds{Q}}
\newcommand {\bC}{\mathds{C}}
\newcommand {\bP}{\mathds{P}}
\newcommand {\bv}{\mathrm{v}}

\newcommand {\bV}{\mathds{V}}
\newcommand{\GL}{\operatorname{GL}}

\newcommand{\Pic}{\operatorname{Pic}}

\def\p#1{{\bP^{#1}}}

\def\mapright#1{\mathbin{\smash{\mathop{\longrightarrow}
\limits^{#1}}}}

\title[Ulrich and instanton bundles on cubic fourfolds]{Ulrich and
  instanton bundles on special cubic fourfolds}



\author[G. Casnati]{Gianfranco Casnati}
\address{
Gianfranco Casnati,
Dipartimento di Scienze Matematiche, Politecnico di Torino,
c.so Duca degli Abruzzi 24,
10129 Torino, Italia.}
\email{gianfranco.casnati@polito.it}

\author[D. Faenzi]{Daniele Faenzi}
\address{Daniele Faenzi.
  Université Bourgogne Europe, CNRS, IMB UMR 5584, F-21000 Dijon, France}
\email{daniele.faenzi@u-bourgogne.fr}

\author[F. Galluzzi]{Federica Galluzzi}
\address{Federica Galluzzi.
 Dipartimento di Matematica,
  Universit\`a di Torino, 
  Via Carlo Alberto 10, 10123 Torino,Italia}
\email{federica.galluzzi@unito.it}

\begin{thanks}{D. F. partially supported by FanoHK ANR-20-CE40-0023, SupToPhAG/EIPHI ANR-17-EURE-0002, Bridges ANR-21-CE40-0017.
}\end{thanks}

\begin{thanks}
{F.G. partially funded by the European Union - Next Generation EU, Mission 4, Component
1, CUP D53D23005860006, within the project PRIN 2022L34E7W  Moduli spaces and birational geometry}
\end{thanks}

\begin{document}

\excludecomment{versionb}

\sloppy

\begin{abstract}
We study instanton and Ulrich bundles on hypersurfaces of the
projective space, with a focus on special cubic fourfolds
and generalized Pfaffians, notably defined by skew-symmetric
endomorphisms of Steiner bundles.
We prove that the acyclic extensions of instantons deform to Ulrich
bundles and deduce that the existence of instantons of low rank and
charge implies the existence of Ulrich bundles of low rank, which in
turn forces the fourfold to lie in some Hassett divisor. Finally we take a
closer look to divisors of cubics with discriminant 18 and 20.
\end{abstract}

\maketitle
\begin{versionb}
   
\label{section:introduction}

\end{versionb}

\section{Introduction}
\label{section:introduction}

Let $X$ be a projective $n$-dimensional variety, $n \ge 2$, polarised by a very ample line bundle $\cO_X(1)$. 
A coherent sheaf $\cE$ on $X$ is Ulrich if 
\[
h^{i-1}(\cE(-i))=h^i(\cE(-i))=0, \qquad \mbox{for $1 \le i \le n$}.
\]
These sheaves, modeled on Cohen-Macaulay modules having the maximal number of minimal generators, were introduced in 
\cite{eisenbud-schreyer} in the analysis of Chow forms and determinantal representations of hypersurfaces. Since then, they have been studied under many perspectives, including derived categories, positivity questions, Boij-S\"oderberg theory, to mention a few. Their existence, conjectured in \cite{eisenbud-schreyer} on any projective scheme, is known in several situations, notably complete intersections, curves and some nice classes of surfaces and threefolds. It remains wide open in general.
We refer to \cite{costa-miro-pons:ulrich} for an exhaustive treatment of Ulrich sheaves.

Relaxing slightly the cohomological condition of Ulrich sheaves, an instanton sheaf on $X$ was defined in \cite{An--Cs1} as a coherent sheaf
$\cE$ such that:
\[h^{i-1}\big(\cE(-i)\big)=h^{i+1}\big(\cE(-i-1)\big)=0, \qquad \mbox{if
  $1\le i \le n-1$},\]
with the extra assumption that $h^1(\cE(-1))=h^{n-1}(\cE(-n))$.
The value of $h^1(\cE(-1))$, frequently denoted by $k$, was called the \textit{charge} of $\cE$. Note that these instanton sheaves are called ordinary in \cite{An--Cs1}.
Up to tensoring with by a line bundle, this gives back earlier definitions of instanton sheaves (or bundles) introduced in many situations, notably projective spaces and Fano threefolds, see for instance \cite{O--S--S,Kuz,faenzi-instanton,antonelli-malaspina-marchesi}. These notions of instantons are inspired on the case of $\p3$, where these bundles arise, via Penrose' twistor transform, from solutions of the Yang-Mills equation on the 4 dimensional sphere $S^4$, see \cite{atiyah-ward, ADHM}.

According to these definitions, an Ulrich sheaf is precisely an instanton of charge $0$. We mentioned that the question of whether any projective variety carries an Ulrich sheaf is not settled: of course, the existence of instantons of any given charge is even less clear. However, for some specific classes of varieties where existence of Ulrich bundles is known (for instance, hypersurfaces) the main questions currently under investigation delve with the possible rank, or more generally Chern characters, of Ulrich and instanton sheaves, as well as the properties of their moduli space, e.g. the dimension, irreducible components, smoothness of such spaces. For instance, non-emptyness and generic smoothness of instanton moduli of high rank and charge has been proved on Fano threefolds of Picard number one in \cite{faenzi-comaschi}, while Ulrich bundles on Del Pezzo threefolds where studied in \cite{ciliberto}.
Very little is known about these questions in dimension $4$ and higher.

Our focus here is mainly on fourfold hypersurfaces $X \subset \bP^5$, polarised by $h_X=c_1(\cO_X(1))$, especially smooth cubic fourfolds.
In this case, Ulrich bundles have been studied in connection with representation theory, derived categories and Hassett divisors, see \cite{faenzi-kim, kim, manivel-ulrich,truong-yen}. To summarize, we know that any $X$ carries Ulrich bundles of rank 6 and that a general $X$ carries Ulrich bundles of any rank of the form $3r$ with $r \ge 2$. Also, some information is known for general fourfolds of some Hassett divisors $\cC_\delta$ in the moduli space of cubic fourfolds, consisting of cubics whose middle integral cohomology contains a sublattice of discriminant $\delta$. For instance if the discriminant equals $8$ (resp. $18$) then $X$ carries an Ulrich bundle of rank $4$ (respectively $3$).

On the instanton side, not much is known for cubic fourfolds. However, we know that instantons of rank 2 are related to Pfaffian-type representations of the hypersurface. More precisely, in \cite{An--Cs2} it is shown that when $X\subseteq\p
{n+1}$ is a smooth hypersurface of degree $d$, the existence of a rank two instanton bundle $\cE$ with charge $\quantum$ on $X$
is equivalent to the existence of a skew-symmetric morphism
\[\varphi\colon\cF(-1)\to\cF^\vee,\]
where $\cF$ is a  \textit{Steiner bundle},
namely $\cF$ has a linear resolution, precisely:
\begin{equation}
  \label{steiner}
  0 \to \cO_{\bP^{n+1}}(-1)^{\oplus k} \to \cO_{\bP^{n+1}}^{\oplus 2d+3k} \to \cF  \to 0,
\end{equation}
such that $X = \bV(\Pf(\varphi))$, the Pfaffian of
$\varphi$, and  $\cE = \coker(\varphi)$.
We say that $X$ is Steiner-Pfaffian of type $\cF^\vee$. Being Pfaffian
amounts to being Steiner-Pfaffian of type $\cO_{\bP^{n+1}}^{2d}$, which
corresponds to the case $k=0$
and $\cE$ is the Ulrich bundle responsible for the
Pfaffian representation. As we said, the focus here is on $n=4$, which is the highest dimension where one hopes to find smooth hypersurfaces which are Pfaffian in our sense.

More generally, when a smooth fourfold hypersurface $X \subset \bP^5$ is obtained as Pfaffian of a skew-symmetric morphism of bundles $\varphi : \cF(-\ell) \to \cF^\vee$, where $\cF$ is any vector bundle on $\bP^5$, one can compute the discriminant of the sublattice in the middle cohomology of $X$ generated by $h_X^2$ and $c_2(\coker(\varphi))$
in terms of $\ell$ and of the Chern classes of $\cF$. These computations are carried out alongside with the numerical data of surfaces $Y \subset X$ obtained from instanton bundles in \S \ref{section:Pfaffians}.

The first main result of this paper is about the deformation theory of instanton sheaves  on a smooth cubic fourfold $X$. Given any instanton sheaf $\cE$ on $X$, we start by considering a tautological extension $\cG$ of $\cE$ by $k$ copies of $\cO_X(1)$. We call $\cG $ the \textit{acyclic extension of $\cE$} by analogy with the terminology of \cite{Kuz}, although the reader should be warned that $\cG(-1)$ is acyclic, not $\cG$ itself.
A key observation is that $\cG (-1)$ lies in a very nice subcategory of the derived category of coherent sheaves on $X$, namely 
the Kuznetsov category $\Ku(X)$. Due to the fact that $\Ku(X)$ is a K3 category, we have that $\Ext^2_X(\cG,\cG)$ is dual to $\Hom_X(\cG,\cG)$ and is thus non-zero.
Also, one sees that  $\Ext^2_X(\cE,\cE)$ is also  non-zero, as it dominates 
$\Ext^2_X(\cG,\cG)$.
If $\Ext^2_X(\cE,\cE)$ is one-dimensional, we say that $\cE$ is \textit{unobstructed}. Relating the deformation theory of $\cE$ to that of $\cG$ and using that the latter is particularly well-behaved, we obtain the following result, see Theorem \ref{unobstructed}.
\begin{theorem} \label{unobstructed-intro}
    Let $\cE$ be an unobstructed instanton sheaf on a smooth cubic fourfold $X$. Then the moduli space of simple sheaves on $X$ is smooth at the point $[\cE]$.
\end{theorem}

The second goal of our paper is to establish a closer relationship between Ulrich and instanton bundles. This is based on the observation that acyclic extensions of instantons have the Chern character of Ulrich bundles. Looking in detail at the deformation theory of $\cG$ we see that $\cG$ deforms flatly to an Ulrich bundle.
This leads to our second main result, see Theorem \ref{theorem ulrich}. 
In the statement, $\fU_X(r)$ denotes the moduli space of simple Ulrich bundles $\cU$ on $X$ of rank $r$, while $\fU_X(r,a)$ is the subspace of Ulrich bundles $\cU$ having $c_2(\cU)^2=a \in \bZ$.

\begin{restatable}{theorem}{instantonvsulrich} 
The space $\fU_X(r,a)$ is either empty or
  smooth and symplectic of dimension 
  \[
  m =
  2+\frac{r^2(3r^2-2r+3)}4-a.
  \]
  
  If $\cE$ is an unobstructed instanton bundle of rank $s$ and charge $k$,
  then $\fU_X(r)$ contains a flat deformation of $\cE \oplus
   \cO_X(1)^{\oplus k}$ and $m=2+r^2-c_2(\cE(-h))^2$, with $r=s+k$.
  If $r<6$ or $3 \nmid r$, $X$ lies in a Hassett divisor $\cC_\delta$ with $\delta$ as in \eqref{delta-intro}.
    All possible
  values of $(r,\delta,m)$ with $r<6$ are in Table
  \ref{table-of-possible}.
\end{restatable}
According to Proposition \ref{list of deltas}, 
a cubic fourfold $X$ carrying an Ulrich bundle $\cU$ with  $\rk(\cU)=r$ and $c_2(\cU)^2=a$ corresponds to a point of $\cC_\delta$, where $\delta$ is computed by
\begin{equation} \label{delta-intro}
\delta = -\frac{r^2(3 r-1)^2}{4} + 3a.    
\end{equation}
It should be noted that, \textit{for given $r$, only a finitely many values of $a$ are allowed}.
Indeed, since Ulrich bundles of rank $r$ have a prescribed Hilbert polynomial, there are finitely many possible choices for the Chern character of the Jordan-H\"older factors of a (necessarily semistable) Ulrich bundle of rank $r$ and each choice determines $a=c_2(\cU)^2$ uniquely. However, since $\delta \ge 0$ and $m \ge 0$, the above formulas provide explicit bounds, which actually force a single choice if and only if $r \in \{2,3\}$.
For $r=4$ already,  we don't know if all choices of $a$ allowed by these bounds are actually realized by a generic cubic of the corresponding divisor $\cC_\delta$.

Anyway, for some choices of $\delta$, a general cubic representing a point in the Hassett divisors $\cC_\delta$ is a natural candidate to support instantons of some charges $k$ and/or Ulrich bundles of some ranks $r$, where $k$ and $r$ depend explicitly on $\delta$.
It should be noted that, if this is the case, in view of the Pfaffian description mentioned above, we automatically get that $\cC_\delta$ is unirational, which is another well-studied  aspect about the moduli space of cubic fourfolds, see for instance \cite{Nuer}.
We look in greater detail at instanton / Ulrich sheaves on cubic fourfolds belonging to Hassett divisors $\cC_{\delta}$ for $\delta=18$ and $\delta=20$.

For $\delta=18$, our main result is Theorem \ref{18-theorem}. It complements and reproves parts of \cite{A--H--T--VA,truong-yen}.
This result outlines a surprising (to us!) connection with Coble cubics, introduced as unique cubic in $\bP^8$ whose singular locus is the Jacobian of a curve of genus 2, embedded by the linear system $|3\Theta|^\vee$. This cubic is connected by projective duality to the moduli space of rank 3 semistable bundles on $C$ -- more precisely, to the branch locus of the theta map toward $|3\Theta|$. In view of \cite{gs}, any Coble cubic $\Gamma$ can be obtained as a Steiner-Pfaffian, starting with an alternating 3-form $\omega$ in $9$ variables, and taking the Pfaffian of the map $\cT_{\bP^8}(-2) \to \Omega_{\bP^8}(1)$ associated with $\omega$. Hence we write $\Gamma=\Gamma_\omega$.
We prove the following result, which is the main portion of Theorem \ref{18-theorem}.
\begin{theorem} \label{C18-intro}
    A generic cubic $X$ representing a point of $\cC_{18}$ carries a 1-dimensional family of instantons $\cE$ of rank 2 and charge 1 and a 2-dimensional families of stable Ulrich bundles of rank 3. Moreover, $X$ is a linear section of a Coble cubic $\Gamma_\omega$ and for each choice of $\cE$ there is a $1$-parameter family of choices for the alternating 3-form $\omega$.
\end{theorem}

For $\cC_{20}$, the natural Steiner bundle to consider is $\cF^\vee = \Omega_{\bP^5}(1)^{\oplus 2}$. Our main result about $\cC_{20}$ is Theorem \ref{20-theorem}, which is summarized as follows.
\begin{theorem} \label{C20-intro}
    A generic cubic $X$ representing a point of $\cC_{20}$ carries a 2-dimensional family of instantons $\cE$ of rank 2 and charge 2 and a 6-dimensional families of stable Ulrich bundles of rank 4. Moreover $X$ is Steiner-Pfaffian of type $\Omega_{\bP^5}(1)^{\oplus 2}$.
\end{theorem}

These results shed some light on the precise relationship between instantons of rank 2 and Ulrich bundles of low rank on some special cubic fourfolds. This complements and (hopefully) clarifies  \cite{faenzi-kim, kim, truong-yen, Nuer} and connects to \cite{An--Cs1, An--Cs2} as well as \cite{Cat} and \cite{gs,bmt}.
As a side note, we mention that we rely on a computer-aided proof via Macaulay 2, \cite{Macaulay2}, when calculating the normal bundle of a surfaces in $\cC_{20}$ associated with an instanton of charge 2.

\begin{ack}
D. F. and F. G. wish to note that this paper grew out of a collaboration started by Gianfranco Casnati, our friend and mentor, shortly before his untimely and unexpected passing. It is dedicated to his beloved memory.

\bigskip
We would like thank J. Pons-Llopis, V. Antonelli and Y. Kim for their helpful comments. D. F. wishes to thank N. Addington for useful discussions about this topic.
\end{ack}

The structure of the paper is as follows. In \S
\ref{section:background} we recall some basic material about
instanton bundles, particularly of rank 2. In 
\S \ref{section:Pfaffians} we discuss Pfaffian
representations of hypersurfaces of $\bP^5$, with a view toward
discriminant lattices in their middle cohomology. More precisely, in \S
\ref{subsection:instantons and surfaces} we deal with instantons and
Steiner-Pfaffian representations of hypersurfaces of type $\cF^\vee$,
where $\cF$ is a Steiner bundle on $\bP^5$, while in 
\S \ref{Pfaffiandiscr} we look more generally at Pfaffian
hypersurfaces of type $(\cF^\vee,\ell)$, where $\cF$ is any vector
bundle on $\bP^5$. 
In \S \ref{section:ulrich} we discuss unobstructed instantons and the relationship between
instanton and Ulrich bundles on a cubic fourfold and show that the
 existence of the former implies the presence of the latter, after
taking acyclic extensions and deformations, under the assumption of unobstructedness.
Next, in \S \ref{section:instanton of rank 2} we analyse the
connection between instanton / Ulrich bundles and some special
Hassett divisors $\cC_\delta$. More precisely, in \S \ref{14:section} we
look at $\cC_{14}$, \S \ref{18-section} we deal with $\cC_{18}$, while 
 $\cC_{20}$ is studied \S \ref{20:section}.
Finally, in \S \ref{section:quartic} we briefly look at some
instantons on quartic fourfolds.

\section{Background}
\label{section:background}

\subsection{Notation and conventions} 

We work over an algebraically closed field $\field$ of characteristic different from 2. By \textit{
  projective variety}, we mean a closed, integral
subscheme  $X$ of some projective space $\bP^N$ over $\field$.
The restriction to $X$ of the hyperplane class $c_1(\cO_{\bP^N}(1))$ will usually be
denoted by $h$ and we say that $(X,h)$ is a \textit{polarised
  variety}.
Sometimes we use this terminology also when $h$ is ample and
$\cO_X(h)$ is globally generated.

An \textit{$n$-fold} is a smooth projective variety of dimension $n$ and a
\textit{surface} is a projective variety over $\field$ of dimension $2$.
We write $\cT_X$ and $\Omega_X$ for the tangent and cotangent sheaves of a projective
variety $X$, they are locally free of rank $n$ if $X$ is an $n$-fold.
In this case $\omega_X\cong\det(\Omega_X)$ and we denote by $K_X$ any divisor such that $\omega_X\cong\cO_X(K_X)$. 
A \textit{hypersurface} in $\bP^{n+1}$ is a projective subscheme of pure dimension $n$. 
If $Y\subseteq X$ is a closed subscheme of a variety $X$ we 
denote by
$\cI_{Y\vert X}$ the ideal sheaf in $\cO_X$ of $Y$. Let us record the
sequence
\begin{equation}
\label{seqStandard}
0\longrightarrow\cI_{Y\vert X}\longrightarrow\cO_X\longrightarrow\cO_Y\longrightarrow0.
\end{equation}

If $Z\subseteq Y\subseteq X$ are closed subschemes, then we have the exact sequence
\begin{equation}
\label{seqChain}
0\longrightarrow \cI_{Y\vert X}\longrightarrow \cI_{Z\vert
  X}\longrightarrow \cI_{Z\vert Y}\longrightarrow 0.
\end{equation}
Given $p \in \bQ[t]$, 
we denote by $\Hilb^{p}(X)$ the Hilbert scheme of closed subschemes
in $X$ with Hilbert polynomial $p$, with respect to the ample divisor
$h$, which we will silently fix.
Given a closed subscheme $Z\subseteq X$ and for all $t \in \bZ$, we denote by $\varrho^Z_t$ the natural restriction map
$$
H^0(\cO_X(t h))\longrightarrow H^0(\cO_Z(t h)).
$$
We say that a $\bk$-linear map $\rho$ has maximal rank, if it is either injective
or surjective. Note that, when $X=\bP^N$ with $N > 1$, $\varrho^Z_t$ has maximal rank if and only if either
$$
h^0(\cI_{Z\vert \bP^N}(t h)) \; \mbox{or} \; h^1(\cI_{Z\vert \bP^n}(t h))=0.
$$ 

Recall that if $Z\subseteq X$ is LCI (locally complete intersection), the conormal
sheaf $\cN_{Z\vert X}$ is locally free and, if $Z$ is smooth, we have
a normal bundle sequence:
\begin{equation}
\label{seqNormal}
0\longrightarrow\cT_Z\longrightarrow\cT_X|_Z\longrightarrow\cN_{Z\vert X}\longrightarrow0.
\end{equation}
If $Z\subseteq Y$ is a smooth closed subvariety, with $Y \subset X$ a closed LCI subscheme, we get
\begin{equation}
\label{seqChainNormal}
0\longrightarrow\cN_{Z\vert Y}\longrightarrow\cN_{Z\vert X}\longrightarrow\cN_{Y\vert X}\vert _Z\longrightarrow0,
\end{equation}
If $\vartheta\colon\cA\to\cB$ is a morphism of locally free sheaves on an $N$-fold $P$  we can define the degeneracy loci
$D(\vartheta)$ as the schemes supported on the set of points
$x \in P$ in which the evaluated morphism $\vartheta(x)$ has not maximal rank.
These have a natural reduced scheme structure induced by their
inclusion in $P$. One can generalizes the previous definition by considering the loci
$$
D_r(\vartheta):=\{x \in P : \rk (\vartheta _x) \leq r\}.
$$
These also have a reduced structure.
If $\vartheta$ is skew-symmetric, then $\rk(\vartheta_x)$ is even at each point $x\in P$. In particular if $r$ is even we know that $D_r(\vartheta)=D_{r+1}(\vartheta)$. Moreover, $D_{r}(\vartheta)$ is set-theoretically the locus of points of $P$ where the $r\times r$ Pfaffians of a local matrix of $\vartheta$ vanish. In order to take care of such a characterization we will denote these degeneracy loci by $\bV(\Pf_r(\vartheta))$. We often omit the subscript $r$ when referring to Pfaffians of maximal size.

\subsection{Reminder on Hartshorne-Serre's construction}

Let $\cA$ be a rank two vector bundle on an $N$-fold $P$ and let $\sigma\in H^0(\cA).$
In general its zero-locus
$D_0(\sigma)\subseteq P$ is either empty or its codimension is at most
$2$. We can always write $D_0(\sigma)=Z\cup S$
where $Z$ has codimension $2$ (or it is empty) and $S$ has pure codimension
 $1$ (or it is empty). In particular $\cA(-S)$ has a section vanishing
on $Z$, thus we can consider its Koszul complex 
\begin{equation}
  \label{seqSerre}
  0\longrightarrow \cO_P(S)\longrightarrow \cA\longrightarrow \cI_{Z\vert P}(-S)\otimes\det(\cA)\longrightarrow 0.
\end{equation}
Sequence (\ref{seqSerre}) tensored by $\cO_Z$ yields  
\begin{equation}
\label{Normal}
\cN_{Z\vert P}\cong\cO_Z\otimes\cA(-S).
\end{equation}
If $S=\emptyset$, then $Z$ is locally complete intersection inside $P$, because $\rk(\cA)=2$. In particular, $Z$ has no embedded components, see \cite[43, Corollaire I.6.11 2) 3)]{Jo}.
The following result reverts the above construction.

\begin{theorem}
  \label{tSerre}
Let $P$ be an $N$-fold with $N\ge2$ and $Z\subseteq P$ a local complete intersection subscheme of codimension $2$.
If $\det(\cN_{Z\vert P})\cong\cO_S\otimes\mathcal L$ for some $\mathcal L\in\Pic(P)$ such that $h^2(\mathcal L^\vee)=0$, then there exists a rank two vector bundle $\cA$ on $P$ satisfying the following properties.
  \begin{enumerate}[label=\roman*)]
  \item $\det(\cA)\cong\mathcal L$.
  \item $Z=D_0(\sigma)$ for some section of $\cA$.
  \end{enumerate}
  Moreover, if $H^1({\mathcal L}^\vee)= 0$, the above two conditions  determine $\cA$ up to isomorphism.
\end{theorem}
\begin{proof}
See \cite[Theorem 1.1.]{Ar}.
\end{proof}

Let $X\subseteq\p N$ be a smooth degree $d$ hypersurface and let $\cO_X(h):=\cO_X\otimes\cO_{\p N}(h)$.
In this case $\cN_{X\vert \p N}\cong \cO_X(dh)$, hence \eqref{seqNormal} with $Z:=X \subseteq P:=\p N$ becomes
\begin{equation}
\label{seqTangentHypersurface}
0\longrightarrow\cT_X\longrightarrow\cO_X\otimes\cT_{\p N}\longrightarrow\cO_X(dh)\longrightarrow0.
\end{equation}
Moreover we also have the Euler sequence
\begin{equation}
\label{seqTangentP5}
0\longrightarrow\cO_{\p N}\longrightarrow\cO_{\p N}(1)^{\oplus N+1}\longrightarrow\cT_{\p N}\longrightarrow0.
\end{equation}
The two  sequences above return
\begin{equation}
\begin{gathered}
\label{ChernTangent}
c_1(\cT_X)=(N+1-d)h=-K_X,\\ c_2(\cT_X)=\left(d^2-(N+1)d+{{N+1}\choose2}\right)h^2. 
\end{gathered}
\end{equation}

Let $X\subseteq\p N$ be a smooth variety of dimension $n\ge1$. If $\cA$ is any torsion-free sheaf we define the {\sl slope of $\cA$ (with respect to $\cO_X(h):=\cO_X\otimes\cO_{\p N}(1)$)} as
$$
\mu(\cA):=\frac{c_1(\cA)h^{n-1}}{\rk(\cA)}.
$$
The torsion-free sheaf $\cA$ is {\sl$\mu$-semistable} (resp. {\sl $\mu$-stable}) if for all subsheaves $\cB$ with $0<\rk(\cB)<\rk(\cA)$ we have $\mu(\cB) \le \mu(\cA)$ (resp. $\mu(\cB)<\mu(\cA)$). 
For any divisor $D$ we denote by $\fM_X(2;D,c_2)$ the moduli space of $\mu$-stable, rank two vector bundles $\cE$ on $X$ with $c_1(\cE)=D$ and $c_2(\cE)h^2=c_2$.

\subsection{Instanton sheaves}\label{subsection:Instantons}

Let us recall the following definition from \cite{An--Cs1}. The instanton sheaves as defined below are called \textit{ordinary} in \cite{An--Cs1}, i.e., in the terminology of loc. cit., their defect $\delta$ is zero.

\begin{definition}
\label{dInstanton}
Let $X$ be a projective variety of dimension $n\ge1$ endowed with an ample and globally generated line bundle $\cO_X(h)$.
A non-zero coherent sheaf  $\cE$ on $X$ is called (ordinary) $h$-instanton sheaf with charge $\quantum\in\bZ$ if the following properties hold:
\begin{itemize}
\item $h^0\big(\cE(-h)\big)=h^n\big(\cE(-nh)\big)=0$;
\item $h^i\big(\cE(-(i+1)h)\big)=h^{n-i}\big(\cE(-(n-i)h)\big)=0$ if $1\le i\le n-2$;
\item $h^1\big(\cE(-h)\big)=h^{n-1}\big(\cE(-nh)\big)=\quantum$.
\end{itemize}
If the polarisation $h$ is clear, we will drop it from the wording. 
\end{definition}

Note that, according to this definition, $\cE$ is an Ulrich sheaf for the polarisation $h$ if and
only if it is an $h$-instanton bundle of charge zero.

\begin{remark}
    A comparison of the above definition against the previous notions of instanton sheaf, notably on projective threefolds, is made in 
    \cite[\S 1]{antonelli-malaspina}, let us summarize the discussion here. Let $X$ be a smooth  Fano threefold  of Picard rank one, so that $\Pic(X)=\langle h_X \rangle$, with $h_X$ ample.
    Let $\cE$ be a slope-semistable $h_X$-instanton sheaf $X$. Let $i_X$ be the index of $X$, so that $\omega_X\simeq \cO_X(-i_Xh_X)$ and write $i_X=q_X+e_X$ with $e_X \in \{1,2\}$. Then $\cE((q_X-2)h_X)$ is an instanton sheaf according to 
\cite{faenzi-instanton, faenzi-comaschi}. 
The converse is true up to some caveat about stability and global sections of $\cE$. Again, we refer to \cite[\S 1]{antonelli-malaspina} for details.

Also, according to some authors, instantons should have a prescribed rank (e.g. in \cite{ancona-ottaviani:instantons}) and may or may not have trivial direct summands.

Again about Fano threefolds, it should be noted that the notion of charge proposed in \cite{faenzi-instanton, faenzi-comaschi}
 may be a bit shifted with respect to the one we give here. For example, when $i_X=2$ and $\cE$ is a slope-semistable instanton of rank $r$ and charge $k$ according to the previous definition, then $\cE(-1)$ is an instanton sheaf of charge $r-k$ according to loc. cit.
\end{remark}

\begin{lemma}
  Let $X$ be a projective variety of dimension $n\ge1$ endowed with an ample and globally generated line bundle $\cO_X(h)$.
Let $\cE$ be an instanton sheaf of rank $r$ and charge $\quantum$. Then
  we have
  \begin{align}
    \label{Slope}
    c_1(\cE) \cdot h^{n-1}& =\frac 12 r((n+1)h-c_1(X)) \cdot h^{n-1}, \\
  \label{Chern}
c_2(\cE) \cdot h^{n-2}&=\quantum + \frac12 c_1(\cE)\cdot (c_1(\cE)+c_1(X)) \cdot h^{n-2} + \\
           \nonumber           &+\frac{1}{12} r
                                 \left(c_1(X)^2+c_2(X)-\frac 12 (3n+2)(n+1) h^2\right)\cdot h^{n-2}.
\end{align}
\end{lemma}

\begin{proof}
The formula concerning $c_1(\cE)$ is obtained by \cite[Theorem 1.6 and Remark 6.5]{An--Cs1}.
About $c_2(\cE)$, we consider a sufficiently general choice of $n-2$
global sections of $\cO_X(h)$. Their base locus is a smooth surface $S
\subset X$ and tensoring the associated Koszul complex by $\cE(-h)$ we
get a long exact sequence: 
\[
  0 \to \cE((2-n)h) \to \cdots \to \cE_X(-ph) ^{\oplus {{n-2} \choose p} }\to
    \cdots  \to \cE(-h) \to \cE(-h)|_S \to 0.
  \]
Using the conditions of Definition \ref{dInstanton} we get
$h^0(\cE(-h)|_S)=0$ and $h^1(\cE|_S(-h))=k$. Also, we get
$h^2(\cE(-2h)|_S)=0$, which in turn implies $h^2(\cE(-h)|_S)=0$.
Therefore $\chi(\cE(-h)|_S)=-k$. 
Then the desired formula is computed
by writing down the Chern classes of $S$ in terms of those of $X$ by the normal bundle sequence
and computing $\chi(\cE(-h)|_S)$ by using Riemann-Roch on $S$.
\end{proof}

\begin{remark}
Note that, if $\Pic(X)$ is generated by $\cO_X(h)$, then \eqref{Slope} implies
$$
c_1(\cE)=\frac{r}2((n+1)h+K_X)
$$
for each $h$-instanton bundle on $X$. For instance if
$X\subseteq\p{n+1}$ is a smooth hypersurface of degree $d$ and dimension $n\ge3$, then 
\begin{equation}
\label{c1}
c_1(\cE)=\frac{r}2(d-1)h.
\end{equation}
By \cite[Remark 3.3]{An--Cs1} and \cite[Corollary 4.3]{An--Cs1}, we also have
\begin{equation}
\label{ChiE}
\chi(\cE(th))=(r h^n+2\quantum){{t+n}\choose n}-\quantum{{t+n+1}\choose n}-\quantum{{t+n-1}\choose n}.
\end{equation}

Moreover $\cE(\quantum h)$ is regular in the sense of Castelnuovo-Mumford by \cite[Proposition 5.1]{An--Cs1}, hence it is globally generated.

\end{remark}

 \section{Pfaffians representations and rank-2 instantons on hypersurfaces}

 \label{section:Pfaffians}

Here we discuss Pfaffian representations of smooth hypersurfaces $X
   \subset \bP^N$ and
   their connection with rank-2 instantons.
   Our main focus is on the case of fourfolds, i. e. $N=5$.
      We say that a skew-symmetric morphism of vector bundle $\varphi : \cF(-\ell) \to
   \cF$ gives a \textit{Pfaffian representation} of $X$ \textit{of type $(\cF^\vee,\ell$)} if the
   equation of $X$ is the Pfaffian of $\varphi$. In this situation, the sheaf
   $\cE=\coker(\varphi)$ is typically a rank-2 instanton bundle on $X$, and the
   datum of this bundle alone allows, under suitable conditions, to
   recover the morphism $\varphi$. Moreover, under the cohomological
   constraints corresponding to the fact that $\cE$ is an instanton
   bundle, the vector bundle $\cF$ takes a specific form, namely $\cF$ is a Steiner bundle, and the
   Pfaffian representation becomes particularly interesting.
   
Throughout the whole section, $X$ is a smooth hypersurface in  $\bP^N$ with $N \ge 4$. Assume that $X$ carries an instanton bundle $\cE$. Then the integral cohomology $H^{2,2}(X,\bZ)=H^{2,2}(X)\cap H^4(X,\bZ)$ inherits the distinguished elements $c_2(\cE)$. 
When $N=5$, the cup product gives $H^{2,2}(X,\bZ)$ a lattice structure and 
it is natural to study the discriminant of the sublattice generated by $h_X^2$ and $c_2(\cE)$. One way to do it is to take a surface $Y$ arising as zero-locus of a global section of $\cE(s)$ for some $s \in \bN$, and to replace $c_2(\cE)$ by $Y$. Otherwise we may compute this directly from the classes of $\cF$ using Grothendieck-Riemann-Roch.
      We deal with instanton bundles of rank 2 and their associated surfaces $Y$ in \S \ref{subsection:instantons and surfaces}, while 
    the more general case of \textit{Pfaffian representation} of $X$ \textit{of type $(\cF^\vee,\ell$)} is treated using the second approach in \S \ref{Pfaffiandiscr}.

 \bigskip
 As we said, we will be mostly interested in the case
  $N=5$ and more precisely on the case $(N,d)=(5,3)$,
  namely when $X$  is a smooth cubic fourfold. We will use the
  definition of instanton given in \S \ref{subsection:Instantons}. It is
  straightforward to check that, if
  $\cE$ is an instanton bundle of rank $r$ and charge $\quantum$ on a
  cubic fourfold $X$, then we have:
  \[
    c_1(\cE(-1))=0.
  \]
  Moreover we have the
  following table for the cohomology of $\cE(t)$.
  \[
    \begin{array}{c||c|c|c|c|c|c}
      t & -5 & -4 &-3 &-2 &-1 & 0 \\
      \hline
      \hline
      h^4(\cE(t)) & \ast & 0 & 0 & 0 & 0 & 0\\
      \hline
      h^3(\cE(t)) & \ast & k & 0 & 0 & 0 & 0 \\
      \hline
      h^2(\cE(t)) & 0 & 0 & 0 & 0 & 0 & 0\\
      \hline
      h^1(\cE(t)) & 0 & 0 & 0 & 0 & k & \ast \\
      \hline
      h^0(\cE(t)) & 0 & 0 & 0 & 0 & 0 & \ast
    \end{array}
  \]
The stars represent some indeterminate values.

\subsection{Instanton bundles on hypersurfaces and associated surfaces}

\label{subsection:instantons and surfaces}

Let $N \ge 4$ be an integer and let $X\subseteq \p N$ be a smooth
degree $d$ hypersurface.
Set $\cO_X(h):=\cO_X\otimes\cO_{\p N}(1)$. For any $t \in \bZ$ and
given a coherent sheaf $\cF$ on $X$, we write freely
 $\cO_X(t)$ for $\cO_X(t h)$ and $\cF(t)$ for $\cF \otimes \cO_X(t)$.
This should generate no confusion, particularly since $\Pic(X)$ is generated by $\cO_X(h)$, see for instance
\cite[Example 3.1.25]{Laz1}.
For the same reason, we remove the reference to the polarisation when discussing instantons
on $X$.

\begin{lemma}
\label{pStable}
A rank-two instanton bundle $\cE$ on $X$ is $\mu$-stable.
\end{lemma}
\begin{proof}
In view of  \cite[Lemma 2.6]{Hop},
it is easy to see that whenever $\cO_X(h)$ generates $\Pic(X)$ and
$\cA$ is a rank two vector bundle on $X$ with $c_1(\cA) \in \{-h,0\}$,
$\mu$-stability of $\cA$ is equivalent to the condition $h^0(\cA)=0$.
Here we have that $\Pic(X)$ is generated by $\cO_X(h)$ and, setting
$q:=\lceil\frac{d-1}2\rceil$, 
also that $c_1(\cE(-qh)) \in \{-h,0\}$. However, we have $q \ge 1$ so the
condition $h^0(\cE(-1))=0$ guarantees that $h^0(\cE(-q))=0$ and $\cE$ is stable.
\end{proof}

\noindent
The following example is an immediate byproduct of the  Theorem \ref{tSerre}.

\begin{example}
\label{eSerre}
Set $N=5$, $d \in \{2, 3\}$ and assume that $X$ contains an isomorphic projection $Y$ of a Del Pezzo
surface of degree $5+\quantum$ inside $\p{5+\quantum}$ for some
integer $\quantum \in\{0,1,2,3,4\}$.
Then, $\det(\cN_{Y\vert X})\cong \cO_Y(5-d)$, hence there is a unique rank two bundle $\cA$ fitting into 
$$
0\longrightarrow\cO_X\longrightarrow\cA\longrightarrow\cI_{Y\vert X}(5-d)\longrightarrow0,
$$
thanks to Theorem \ref{tSerre}. Let $\cE:=\cA(d-3)$, hence $c_1(\cE)=(d-1)h$. Notice that
$$
h^0(\cE(-1))=h^0(\cA(d-4)=0
$$
because $Y$ is non-degenerate and
$$
h^1(\cE(-2))=h^1(\cA(d-5))=h^1(\cI_{Y\vert X})=0.
$$
Moreover, the above sequence tensored by
$\cO_X(d-6)$ and  \eqref{seqStandard} tensored by $\cO_X(-1)$
give 
$$
h^2(\cE(-3))=h^2(\cA(d-6))=h^2(\cI_{Y\vert X}(-1))=h^1(\cO_Y(-1))=0.
$$
Arguing similarly
$$
h^1(\cE(-2))=h^1(\cA(d-4))=h^1(\cI_{Y\vert X}(1))=\quantum.
$$
By  \cite[Proposition 6.7]{An--Cs1} we deduce that $\cE$ is a rank two instanton bundle with charge $\quantum$.
\end{example}

Let us summarize the situation about surfaces obtained as zero-loci of
twisted sections of rank two instantons. We set $H^{2,2}(X,\bZ):=H^{2,2}(X)\cap H^4(X,\bZ)$.
In the statement, according to the standard terminology, a surface is called \textit{minimal} if it contains no $(-1)$-curves.

\begin{proposition}
\label{pSurfaceY}
Set $N=5$ and assume that $\cE$ is a rank two instanton bundle with charge $\quantum$ on $X$ and let $\sigma\in H^0(\cE(s))$ be such that its zero-locus $Y:=(\sigma)_0\subseteq X$ has pure dimension $2$. Then the following assertions hold.
\begin{enumerate}
\item $Y$ is a locally complete intersection surface in $X$ and $\varrho^Y_t$ has maximal rank if  $t\le d+s-2$ and $t\ge \quantum+d+s-2$.
\item If $d+s\ge3$, then $Y$ is connected, non-degenerate and $q(Y)=0$. Moreover, it is linearly normal unless possibly if $d+s=3$ and $\quantum\ge1$: in this case
\begin{equation}
\label{Normal1}
h^0(\cO_Y(1))=h^0(\cO_{\p5}(1))+\quantum.
\end{equation}
\item If $d\ge3$, then $Y$ is not contained in any quadric unless possibly if $d+s=3$: in this case
\begin{equation}
\label{Quadric}
h^0(\cI_{Y\vert\p5}(2))=\left\lbrace\begin{array}{ll} 
h^0(\cE)-1
\quad&\text{if $s=0$,}\\
h^0(\cE)\quad&\text{if $s\ne0$.}
\end{array}\right.
\end{equation}
\item The canonical line bundle of $Y$ is 
\begin{equation}
\label{Omega}
\omega_Y\cong\cO_Y(2d+2s-7),
\end{equation}
hence $Y$ is minimal, if $d+s\ge4$.
\item The degree of $Y$ is 
\begin{equation}
\label{Degree}
\deg(Y)=ds^2 +d(d-1)s+\frac d6(2d^2-3d+1)+k.
\end{equation}
\item The Euler characteristic of $Y$ is
\begin{equation}
\label{Chi}
\begin{aligned}
\chi(\cO_Y)&=\frac{7}{12}ds^4+
\frac d6(10d-{25})s^3+
\frac d{12}(22d^2-102d+{115})s^2\\
&+
\frac d{12}({11}d^3-70d^2+{137}d-78)s\\
&+
\frac d{120}(22d^4   -175d^3    +460d^2    -425d   +118)\\
&+\frac12k(s^2+(2d-7)s+d^2-7d+12).
\end{aligned}
\end{equation}
\item The self-intersection of $Y$ in  $H^{2,2}(X,\bZ)$ is

\begin{equation}
\label{Self}
\begin{aligned}
Y^2&=ds^4+2d(d-1)s^3+
\frac d3(5d^2-9d+4)s^2\\
&+\frac d3(2d^3 -5d^2+4d -1)s+
\frac d{30}(4d^4   -10d^3  +10d^2   -5d   +1)\\
&+\quantum(2s^2+2(d-1)s+d^2-d-1).
\end{aligned}
\end{equation}
\end{enumerate}
The surface $Y$ is actually smooth (in characteristic zero) if $\sigma$ is sufficiently general
and $\cE(s)$ is globally generated, which happens for instance if $s=k$.
\end{proposition}
\begin{proof}
We first notice that $s\ge0$, because  $h^0(\cE(-1))=0$ by definition.
Moreover, we know that $Y$ has pure dimension $2$ by hypothesis, hence it is a locally complete intersection subscheme of $X$ and there exists an exact sequence of the form
\begin{equation}
\label{seqIdeal}
0\longrightarrow\cO_X\longrightarrow\cE(s)\longrightarrow\cI_{Y\vert X}(d+2s-1)\longrightarrow0,
\end{equation}
thanks to Theorem \ref{tSerre} and \eqref{c1}. Thus the cohomology of \eqref{seqIdeal} tensored by $\cO_X(1-d-2s+t)$ yields

\[
0\to H^0(\cO_X(1-d-2s+t))
\to H^0(\cE(1-d-s+t))\to H^0(\cI_{Y\vert X}(t))\to0 
\]
Also $h^1(\cI_{Y\vert X}(t))= h^1(\cE(1-d-s+t))$.
We infer that $h^1(\cI_{Y\vert X}(t))=0$ in the ranges $i=2$ and $t\le d+s-3$, $i=1$ and $t\ge \quantum+d+s-2$, $i=0$ and $t\le d+s-2$, because $\cE$ is an instanton. 

If $\cO_Y(h_Y):=\cO_Y\otimes\cO_{\p5}(1)$, then the cohomology of \eqref{seqStandard}
and the isomorphisms $ H^0(\cO_{\p5}(t))\cong H^0(\cO_{X}(th))$ imply that $\varrho^Y_t$ has maximal rank in the same range.

Taking into account the above vanishings, $h^i(\cI_{Y\vert X})=0$ for $i\le2$ when $d+s\ge3$, hence $q(Y)=0$ and $h^0(\cO_Y)=1$: consequently $Y$ is connected. Similarly $h^0(\cI_{Y\vert X})=0$ when $d+s\ge3$, hence $Y$ is non-degenerate. When $d+s\ge4$, then $h^1(\cE(1))=0$ whence $Y$ is also linearly normal. When $d+s=3$, then $h^1(\cE(2-d-s))=\quantum$, whence \eqref{Normal1}. 

If $d\ge3$, writing  \eqref{seqChain} in our setting and taking
cohomology yields
$h^0(\cI_{Y\vert \p5}(2))=h^0(\cI_{Y\vert X}(2))$, hence \eqref{Quadric} follows trivially.

Thanks \eqref{Normal} we have $\cO_Y\otimes\cE\cong\cN_{Y\vert X}$, hence the adjunction formula on $X$ and \eqref{ChernTangent} with $N=5$ yield
$$
\omega_Y\cong\det(\cN_{Y\vert X})\otimes\omega_X\cong\cO_Y(2d+2s-7),
$$
which is \eqref{Omega}. Notice that \eqref{Omega} and the Nakai criterion imply the minimality of $Y$ when $2s+2d\ge7$, thanks to \cite[Proposition III.2.2]{B--H--P--VV}.

The class of $Y$ in $H^{2,2}(X,\bZ)$ is $c_2(\cE(s))$, whence
$$
\deg(Y)=c_2(\cE(s))h^2=c_2(\cE)h^2+s c_1(\cE)h^3+s^2 d
$$
because $h^4=d$. Thus \eqref{Chern} and \eqref{ChernTangent} with $N=5$ yield \eqref{Degree}.

We have $Y^2=c_2(\cN_{Y\vert X})$ (e.g. see \cite[Corollary 6.3]{Ful}). Thanks to \eqref{seqNormal} with $Z:=Y$ and $P:=X$, a standard Chern classes computation and \eqref{ChernTangent} with $N=5$ combined with the Noether formula $c_2(\cT_Y)=12\chi(\cO_Y)-K_Y^2$ return
$$
Y^2=(d^2-6d+15)h_Y^2+(6-d)h_YK_Y+2K_Y^2-12\chi(\cO_Y).
$$
On the one hand, $h_YK_Y$ and $K_Y^2$ are multiples of $h_Y^2=\deg(Y)$ thanks to \eqref{Omega}. On the other hand, the cohomologies of \eqref{seqIdeal} tensored by $\cO_X(1-d-2s+1)$ and \eqref{seqStandard} yield 
$$
\chi(\cO_Y)=\chi(\cO_X)-\chi(\cE(1-d-s))+\chi(\cO_X(1-d-2s).
$$
Taking into account \eqref{ChiE} we obtain \eqref{Chi}. Thus, we finally obtain \eqref{Self} by combining \eqref{Chi}, \eqref{Omega}, \eqref{Degree}.
\end{proof}

The last statement of the previous proposition is clear due to
Bertini's theorem, in view of the next remark.

\begin{remark}
\label{rFirstSection}
In general, it is not clear that a section of $H^0(\cE(s))$ vanishes
on a scheme of pure dimension $2$ or, even when it does,
such subscheme $Y$ may fail to be smooth or even reduced.
Nevertheless, $Y$ has pure dimension $2$ if we take
$$
s=s_0:=\min\{\ t\in \bZ\ \vert\ h^0(\cE(t))\ge1\ \}
$$
because $\Pic(X)$ is generated by $\cO_X(1)$. Notice that $s_0\ge0$ by definition. 
Now define
$$
s_1:=\min\{\ t\in \bZ\ \vert\ \cE(t)\ \mbox{is globally generated}\ \}.
$$
Then taking $s=s_1$ and the section $\sigma$ general enough, the vanishing locus $Y$ will be smooth.
We observe that, if $\cE$ has charge $\quantum$, then we know that $\cE(k)$ is regular, hence globally generated. Thus, we have $s_0\le s_1 \le \quantum$. 
\end{remark}

\begin{corollary}
\label{cBound}
If $\cE$ is a rank two instanton bundle with charge $\quantum$ on $X$, then

\begin{equation}
\label{Bound}
0\le \quantum\le \frac{d}{30}\left(5(d^2-4)+3\sqrt{5d^4-25d^2+45}\right).
\end{equation}
Moreover, equality occurs on the right if and only if there is a section of $\cE(\quantum h)$ whose zero-locus is homologous to the complete intersection of $X$ with two further hypersurfaces.
\end{corollary}
\begin{proof}
The restriction of the intersection product on $H^{2,2}(X,\bZ)$  to the sublattice generated by $h^2$ and $Y$ is positive semidefinite (and actually positive definite if and only if $Y$ is not homologous to a complete intersection), thanks to the Hodge index theorem. Thus we have
$$ \delta = 
\left\vert
\begin{array}{cc}
h^4&h^2Y\\
h^2Y&Y^2
\end{array}\right\vert\ge0.
$$
Plugging $s=k$ into  \eqref{Degree} and \eqref{Self}, using $h^4=d$ we obtain
$$ \delta = 
4d^6 -5d^4 +d^2 +60d(d^2-4)k -180k^2\ge0.
$$
A simple computation yields \eqref{Bound}, because $\quantum=h^1(\cE(-1))\ge0$.
\end{proof}

\begin{remark}
    The above corollary shows that, on any fourfold hypersurface, the charge $k$
of rank-2 instanton bundles is bounded. This should be contrasted with the situation on threefold hypersurfaces,
where it can be arbitrarily large.
For the construction of instanton bundles of arbitrarily high charge, we refer to \cite{faenzi-instanton} in the case of rank 2 and to \cite{faenzi-comaschi} for ranks $r \ge 3$.
\end{remark}

\begin{example}
Let $X\subseteq\p5$ be a smooth quadric hypersurface. In this case \eqref{Bound} becomes $0\le \quantum\le 1$. If $\quantum=0$  we have the twists of the spinor bundles $\cS'$ and $\cS''$. 

Thus the only non-trivial case could be $\quantum=1$. Since the ideal of an isomorphic projection $Y\subseteq\p5$  of a smooth Del Pezzo surface of degree $6$ in $\p6$ is contained in a smooth quadric hypersurface (see \cite[Lemma 3.7]{KapG}), Example \ref{eSerre} shows that such case actually occurs.

Thanks to \cite[Example 7.5]{An--Cs1} we know that each rank two instanton on $X$ with charge $\quantum=1$ is the cohomology of a monad of the form
\begin{equation}
\label{MonadEven}
0\longrightarrow\cO_{X}\longrightarrow\cS'(1)\oplus\cS''(1)\longrightarrow\cO_{X}(1)\longrightarrow0,
\end{equation}
These are exactly the bundles mentioned in \cite{Ott3}.

Since there are no rank two instanton bundles with charge $\quantum\ge2$ on a smooth quadric hypersurface in $\p5$, it follows from Example \ref{eSerre} that each isomorphic projection in $\p5$ of a Del Pezzo surface of degree at least seven is not contained in any smooth quadric hypersurface in $\p5$.
\end{example}

\begin{example}
\label{ePfaffian}
Let $d\ge3$ and consider a skew-symmetric matrix $\varphi$ of order $2d$ with linear entries in $\field[x_0,\dots,x_5]$: they form a vector space of dimension
$$
6{{2d}\choose2}=12d^2-6d.
$$The Pfaffian of $\varphi$ defines a degree $d$ hypersurface $X\subseteq\p5$ whose singular locus is contained inside the locus of points where $\rk(\varphi)\le 2d-4$. If $\varphi$ is general such a locus has codimension $6$ inside $\p5$, hence it is empty, i.e. $X$ is also smooth for a general choice of $\varphi$. 

Thanks to the results in \cite[Section 8]{Bea} this is equivalent to the existence of a rank two Ulrich bundle on $X$. Thus there always exist smooth degree $d$ hypersurfaces in $\p5$ supporting instanton bundles with charge $\quantum=0$.

Notice that the locus of Pfaffian degree $d$ hypersurfaces $X\subseteq\p5$ has dimension $8d^2-6d=12d^2-6d-\dim(\GL_{2d})$ which is smaller than
$$
\dim\vert\cO_{\p5}(d)\vert={{d+5}\choose5}-1
$$
for $d\ge3$: for further details see \cite[Sections 8 and 9]{Bea}.
\end{example}

\begin{lemma} \label{normale-ext}
Let $\cE$ is a rank-2 instanton bundle on a
hypersurface $X \subset \bP^5$ of degree $d$ and let $\sigma\in
H^0(\cE(s))$. Assume that $Y:=(\sigma)_0\subseteq X$
has pure dimension $2$. Then 
\[
  \dim(\Ext^2_X(\cE,\cE)) \le h^1(\cN_{Y|X}).
\]
\end{lemma}

\begin{proof}
  Recall that $c_1(\cE)=(d-1)h$, so $\cE^\vee \simeq \cE(1-d)$. Under
  our assumptions, we can use \eqref{seqIdeal}. Tensoring with $\cE^\vee(-sh)$, we get:
  \[
    0 \to \cE(1-d-s) \to \cE^\vee \otimes \cE \to \cE(s) \otimes \cI_{Y|X} \to 0.
  \]
  Since $\cE(s)_{|Y} \simeq \cN_{Y|X}$, we also get the exact sequence:
  \[
    0 \to \cE(s) \otimes \cI_{Y|X}  \to \cE(s) \otimes \cO _Y \to \cN_{Y|X} \to 0.
  \]
  The instanton condition of Definition \ref{dInstanton} implies
  $h^2(\cE(1-d-s)=h^2(\cE(s))=0$. 
  Then:  
  \[
    H^1(\cN_{Y|X}) \twoheadrightarrow H^2(\cE(s) \otimes \cI_{Y|X}) \hookleftarrow \Ext^2_X(\cE,\cE).
  \]
  This implies the desired inequality.
\end{proof}

\begin{remark} If $c_1$ and $c_2$ are as in \eqref{c1}  and in \eqref{Chern}, we will denote by $\fI_X(2;\quantum)\subseteq\fM_X(2;(d-1)h,c_2)$ the locus of rank two instanton bundles with charge $\quantum$. 
In the next sections we inspect the loci $\fI_X(2;\quantum)$ when on hypersurfaces $X$ of degrees $d\ge3$. 
Though $\quantum$ is an integer satisfying \eqref{Bound}, such a  locus might be empty: e.g. if $d=3$, then $\fI_X(2;1)=\emptyset$ if $X$ does not represent a point in $\cC_{18}$. It is then interesting to understand the dimension of $\fI_X(2;\quantum)$ when it is not empty for every $d\ge3$ and $\quantum\ge0$. 
\end{remark}

\subsection{Pfaffians and discriminants}\label{Pfaffiandiscr}

We set the ambient dimension $N$  to
$5$ for this section.
If $N > 5$, a hypersurface admitting a Pfaffian
representation will typically be singular, except for some very special
situations, so the choice $N=5$ represents the highest dimension where smooth Pfaffian
hypersurfaces can be studied in general.
Let $\cF$ be a vector bundle of rank $2r$ on $\bP^5$ with
$c_i(\cF)=e_ih^i$, for some $(e_1,\ldots,e_5) \in \bZ^5.$ Let $\ell \in \bZ$.
Consider a skew-symmetric map
$$\varphi :
\cF(-\ell) \to \cF^\vee.$$
Assume that $X=\bV(\Pf_r(\varphi))$ is a smooth hypersurface of degree $d$.
Recall that, in this case, we say that \textit{$X$ is Pfaffian of type $(\cF^\vee, \ell)$}.

If this happens, the support of $\coker(\varphi)$ is contained in $X$.
Indeed, if $\Pf_r(\varphi)$ is non-zero at $x \in \bP^5$, then $\coker(\varphi)$ vanishes at $x$, so that any homogeneous polynomial lying in the annihilator of $\coker(\varphi)$ is a multiple of $\Pf_r(\varphi)$.
Moreover, at all points $x \in X$ the map $\varphi$ has corank
$2$, for otherwise all Pfaffians of order $r-1$ of $\varphi$ vanish at $x$, so that $X$ would be singular at the point $x$.
Then we have
\[
  \coker(\varphi) \simeq i_*(\cE),
\]
for some vector bundle $\cE$ of rank $2$ on $X$, where $i : X \to
\bP^5$ is the inclusion.
We have:
\[
  c_1(\cE)=(d-\ell)h_X,
\]
where $h_X$ is the restriction to $X$ of the hyperplane class of $\bP^5$.
We want to prove the following statement,
\begin{lemma}
\label{delta}
If $X$ is a smooth Pfaffian hypersurface of type $(\cF^\vee,\ell)$ and
degree $d$ in $\bP^5$
then there is a sublattice of $H^{2,2}(X,\bZ)$ having discriminant
$\delta$, where $\delta$ is a polynomial function
of $\ell, r$ and of the Chern classes of $\cF$.
\end{lemma}

\begin{proof}
To begin with, we set
\[
  u_0 = \frac 16 r(r-1)(2r-1)\ell^3, \qquad
  u_1 = -r(r-1)\ell^2, \qquad 
  u_2 = (r-1)\ell.
\]
Applying Grothendieck-Riemann-Roch to the map $i$, we compute:
\[
  c_2(\cE) \cdot h_X^2 = \sum_{0 \le i \le 2 } u_id^i + (d+\ell -\ell r) e_2 +e_3.
\]
The computation of $c_2(\cE)^2$ can also be carried out using Grothendieck-Riemann-Roch, in the
following way.
We first define some polynomial functions of $r$ and $\ell$.
\begin{align*}
  v_0 &= -\frac 1{30}r(r-1)(2r-1)(3r^2-3r-1)\ell^5,   &v_1 &= \frac 16 r(r-1)(6r^2-8r+1)\ell^4,\\
  v_2 &= -\frac 16 r(r-1)(10r-11)\ell^3,  &v_3 &= (r-1)^2\ell^2.
\end{align*}
We also define the following coefficients.
\begin{align*}
 w_3  = &((2r-3)d-(r-1)^2\ell)\ell, & w_2 & = (2d^2(r-1)-d(r-1)(3r-2)\ell+r(r-1)^2\ell^2)\ell,\\
   w_4  = &\ell r-d-2\ell, & w_5 &= -1, \\
  w_{2,2}  =& - \ell r + d + \ell,&   w_{2,3} & = 1 .
\end{align*}

For any multiset $I = \{i_1,\ldots,i_k\}$
whose support $|I|$ is in $\llbracket 2,5 \rrbracket = \{2,\ldots,5\}$
we set $e_I=e_{i_1}\cdots e_{i_k}$. We consider the coefficient $w_I$
for any such multiset, where any
coefficient $w_I$ not already defined in the previous display is tacitly set to 0.
The result of applying Grothendieck-Riemann-Roch to the map $i$ is:
\[
  c_2(\cE)^2 = \sum_{0 \le j \le 3}v_id^j +   \sum_{|I| \subset \llbracket 2,5 \rrbracket }w_Ie_I.
\]
We define now the following coefficients.
\begin{align*}
  a_0&=\frac {-1}{36}r^2(r-1)^2(2r-1)^2\ell^6, &
a_1&=\frac 1{30} r(r-1)(2r-1)(7r^2-7r+1)\ell^5, \\
  a_2&=\frac {-1}6 r(r-1)(2r-1)^2\ell^4, &
a_3&=\frac 1 6 r(r-1)(2r-1)\ell^3. \\
  b_2 &= \frac 13 r(r-1)(\ell r-d-\ell)(2\ell r-3d-\ell)\ell^2, &
b_3 & = \frac {-1}3 (r(r-1)(2r-1)\ell^2+3d(d+\ell-\ell r^2)\ell), \\
  b_4 &= d(\ell r-d-2\ell),
& b_5 & = -d.
\end{align*}
Finally, we put:
\begin{equation*}
  b_{2,2} =  -(r-1)(\ell r-d-\ell)\ell, \qquad 
  b_{2,3} = 2\ell r-d-2\ell, \qquad 
  b_{3,3}=-1.
\end{equation*}
Again we consider multisets $I$ with $|I| \subset \llbracket 2,5
\rrbracket$ and put $b_I=0$ if $b_I$ is not defined in the previous
displays. In conclusion, we obtain the following formula for discriminant of the sublattice
generated by $c_2(\cE) \cdot h_X^2$ and $c_2(\cE)^2$.
\begin{equation}
  \label{delta-formula}
  \delta = \sum_{0 \le j \le 3} a_jd^j +
  \sum_{|I| \subset \llbracket 2,5 \rrbracket }b_Ie_I.
\end{equation}
\end{proof}

For instance, if $X$ is a linear Pfaffian of degree $d$, so that $\cF$ is the
trivial bundle of rank $d$ (hence $e_i=0$ for all $i\ge 0$) and $\ell=1$, we get
\[
  \delta = \frac {1}{180} d^2(d^2-1)(4d^2-1).
\]

\section{Instantons and Ulrich complexity of special cubic fourfolds}

\label{section:ulrich}

In this section we look at the interplay between instantons (notably
rank-2 instantons) and Ulrich bundles of low rank on special cubic
fourfolds.
Indeed, on one hand, it was proved in \cite{faenzi-kim} that any
smooth cubic fourfold $X$ carries an Ulrich bundle of rank $6$.
Also, the smallest rank of an Ulrich bundle on $X$, also known as the Ulrich
complexity $\uc(X)$, equals $6$ if $X$ is very general.
On the other hand, $\uc(X)$ may be smaller than $6$ if $X$ is special.
To recall the terminology, we let $\mathfrak X_3$ be the open subset
of $\vert\cO_{\p5}(3)\vert$ parameterizing
smooth hypersurfaces. The locus $\mathfrak X_3$ is contained in the set of stable points with respect to the natural action of the group $\PGL_6$ on $\vert\cO_{\p5}(3)\vert$ (see \cite[Proposition 4.2]{M--F--K}. The geometric quotient  
$\cC:=\mathfrak X_3/\PGL_6$
is the moduli space of smooth cubic hypersurfaces in $\p5$.
There is a natural map $\mathfrak c\colon \mathfrak X_3\rightarrow \cC$. 
A hypersurface $X$ is called {\sl special} if it contains a smooth
surface which is not homologous to a complete intersection inside $X$.
As proved in \cite[Theorem 1.0.1]{Has}, the set of special smooth
cubic hypersurfaces in $\cC$ is a union of
a countable family of irreducible divisors $\cC_\delta$
where $\in\bZ$, $\delta\equiv0,2\pmod 6$, $\delta\ge8$.
Such loci have been studied, at least for low values of $\delta$, for instance we have

\begin{itemize}\item[$\cC_8$.] Each point  represents a hypersurface containing a plane: see \cite[Example 4.1.1]{Has}.
\item[$\cC_{12}$.] The general point represents a hypersurface containing a cubic scroll: see \cite[Example 4.1.2]{Has}.
\item[$\cC_{14}$.] The general point represents a hypersurface containing a rational normal quartic scroll (or, as alternative description, a quintic Del Pezzo surface in $\p5$): see \cite[Example 4.1.3]{Has}.
\item[$\cC_{18}$.] The general point represents a hypersurface containing an elliptic sextic scroll (or, as alternative description, the isomorphic projection in $\p5$  of a smooth Del Pezzo surface of degree $6$ in $\p6$): see \cite[Section 1]{A--H--T--VA}.
\item[$\cC_{20}$.] The general point represents a hypersurface containing a Veronese surface $\p5$: see \cite[Example 4.1.4]{Has}.
\end{itemize}

\noindent
In what follows, we will set $\hat{\cC}_{\delta}:=\mathfrak c^{-1}(\cC_\delta)\subseteq \mathfrak X_3$.

The connection between Hassett divisors and Ulrich complexity is
apparent in some cases. For example, $X$ is Pfaffian if and only if
$\uc(X)=2$ and this happens if $X$ is general in $\hat{\cC}_{14}$. We refer to \cite{bolognesi-russo-stagliano} for a detailed analysis of this.

This connection has been
further explored more recently in \cite{kim}.
It is shown there that for only finitely many values of $\delta$
a generic cubic $X$ in $\hat{\cC}_\delta$ may satisfy $\uc(X)=4$ and that
$\uc(X)=4$ for $X$ generic in $\hat{\cC}_8$. Also, if $\uc(X)=3$, then $X$
lies in  $\hat{\cC}_{18}.$
We will review this fact below, which is however somehow clear from
\cite{kim}, while the existence of a family of stable Ulrich bundles of rank
$3$ parametrized by a symplectic surface
supported on a general fourfold of $\hat{\cC}_{18}$  was proved in \cite{truong-yen}.

Here, we attempt to clarify the situation a bit. Indeed, we will
start by listing all
possible values of $\delta$ such that a generic fourfold $X$ in
$\hat{\cC}_\delta$ may have $\uc(X) < 6$.
Then, we will see how Ulrich bundles can be obtained from instantons
by deforming their acyclic extensions.
This will provide a way to construct Ulrich bundles of minimal rank on
a cubic lying in $\hat{\cC}_\delta$, for  $\delta=18$ and $\delta=20$,
which we will do in the next section, specifically focused on rank-2
instantons.

Let us state the first main result of this section. Let $X$ be a smooth cubic fourfold and let us denote by 
$\fU_X(r)$ the moduli space of simple Ulrich bundles of rank $r$
on $X$. We write $\fU_X(r,a)$ for the subspace of $\fU_X(r)$ consisting of  bundles
$\cU$ with $c_2(\cU)^2=a$.

\begin{theorem} \label{theorem ulrich}
  When not empty, the space $\fU_X(r,a)$ is
  smooth and symplectic in dimension $m =
  2+(r^2(3r^2-2r+3))/4-a$.
  If $\cE$ is an unobstructed instanton of rank $s$ and charge $k$ on $X$
  then $\fU_X(r)$ contains a deformation of $\cE \oplus
  \cO_X(1)^{\oplus k}$ and $m=2+r^2-c_2(\cE(-1))^2$, with $r=s+k$.
  If $r<6$, or $3 \nmid r$, $X$ lies in a Hassett divisor $\cC_\delta$, where for $r \le 5$ the possible
  values of $(r,\delta,m)$ are classified in Table
  \ref{table-of-possible}.
\end{theorem}

Note that, in Table \ref{table-of-possible}, the dimension of families
of Ulrich bundles arising from deformations of $\cE \oplus
\cO_X(1)^{\oplus k}$ as in the above theorem are distinguished by a
bold letter.
The proof of the above theorem will result from Proposition \ref{list
  of deltas} and Proposition \ref{deforms to Ulrich}, together with
the straightforward computation of $m-2=(r^2(3r^2-2r+3))/4-a$ as
$r^2-c_2(\cE(-1))^2$ where $a=c_2(\cE \oplus \cO_X(1)^{\oplus k})^2$.

To explain the meaning of \textit{unobstructed}, let us mention the fact that, given a (simple) instanton sheaf $\cE$ on a cubic fourfold $X$, we have $\dim \Ext^2_X(\cE,\cE) \ge 1$, we will see this in Lemma \ref{at-least-1}, cf.  in particular \eqref{one-d}. 
When equality is attained, the deformation theory of $\cE$ is particularly nice. For this reason, we give the following definition.

\begin{definition}
      Let $\cE$ be an instanton sheaf on a smooth cubic fourfold $X$. We say that $\cE$ is \textit{unobstructed} if $\Ext_X^2(\cE,\cE)$ is 
  one-dimensional.
\end{definition}

This implies that the relevant moduli space of sheaves is smooth at $\cE$. Indeed, we will prove the following result, which is the second main point of this section.
\begin{theorem}
     \label{unobstructed}
    Let $\cE$ be a simple unobstructed instanton sheaf. Then the moduli space of simple instanton sheaves on $X$ is smooth at the point $[\cE]$.
\end{theorem}

Theorem \ref{theorem ulrich} will be proved in the next two sections, while Theorem \ref{unobstructed} will be shown in \S \ref{section:unobstructed}.

\subsection{Ulrich bundles and the Kuznetsov component}

An Ulrich bundle $\cU$ on a cubic fourfold $X$ is characterized by the
cohomological vanishing:
\begin{equation*}
  H^*(\cU(-t))=0, \qquad \mbox{$t \in 1,\ldots,4$.}  
\end{equation*}
Incidentally we note that, if $\cU$ has rank $r$, we have
\[
  \chi(\cU(t))=\frac {r}{8} \prod_{1\le j \le 4} (t+j).
\]

We will use the fact that Ulrich bundles lie in the Kuznetsov category
of $X$.
Recall that the derived category $\Db(X)$ of bounded complexes of coherent sheaves
on a smooth cubic fourfold $X$ has the semiorthogonal decomposition
\[
  \Db(X) = \langle \Ku(X),\cO_X,\cO_X(1),\cO_X(2) \rangle
\]
where the full triangulated subcategory $\Ku(X)$, called the
\textit{Kuznetsov component} of $X$, consists of the objects $\cE$ of
$\Db(X)$ such that $H^*(\cE)=H^*(\cE(-1))=H^*(\cE(-2))=0$, see \cite{kuznetsov:V14}.
The Kuznetsov component is a K3 category, namely, for any pair of
objects $\cE$ and $\cF$ of $\Ku(X)$, we have a natural isomorphism
\[
  \Hom(\cE,\cF) \simeq \Hom(\cF,\cE[2])^\vee. 
\]
Note along the way that a coherent sheaf $\cU$ on $X$ is an Ulrich bundle if and
only if $\cU(-1)$ and $\cU(-2)$ both lie in $\Ku(X)$.

\begin{proposition} \label{list of deltas}
  Let $X$ be a smooth cubic fourfold and $\cU$ be a simple Ulrich
  bundle of rank $r > 0 $ on $X$.
  Then $\cU$ belongs to a smooth symplectic family of
  dimension
  \[
    m = 2+\frac{r^2(3r^2-2r+3)}{4}-c_2(\cU)^2.      
  \]
  Also, if $r$ is prime to $3$, then $X$ lies in $\hat{\cC}_\delta$, with
  \begin{equation}
    \label{eq:delta}
    \delta = -\frac{r^2(3 r-1)^2}{4} + 3c_2(\cU)^2.
  \end{equation}
  Next, assume that $\uc(X) < 6$.
  Then $X$ belongs to $\hat{\cC}_\delta$ with
  \[
    \delta \in \{8, 14, 20, 26, 32, 38, 44, 50, 56\}.
  \]
  More precisely, if $r \in \{2,3,4,5\}$,
  the possible values of $(r,\delta,m)$ are as in the following table,
  expressing $m$ as a function of $r,\delta$:
\end{proposition}

  \begin{table}[H]
    \centering
    \renewcommand{\arraystretch}{1.1}
    \begin{tabular}{c||c|c|c|c|c|c|c|c|c|c}
      \toprule
      {$r \backslash \delta$} & 8 & 14 & 18 & 20 & 26 & 32 & 38 & 44 & 50 & 56 \\
      \midrule
      \midrule
      2 &    & \textbf{0} &    &    &    &    &    &    &    &    \\
      \midrule
      3 &    &            & \textbf{2} &    &    &    &    &    &    &    \\
      \midrule
      4 & 10 & 8          &            & \textbf{6} & 4  & 2  & 0  &    &    &    \\
      \midrule
      5 & 16  & 14         &            & \textbf{12}& 10 & 8  & 6  & 4  & 2  & 0  \\
      \bottomrule
    \end{tabular}
    \caption{\textit{Values of $m=\dim(\fU_X(r))$ for $r<6$}.}
    \label{table-of-possible}
  \end{table}
  
  \begin{remark}
    We do not know if all the values of the above table are actually
    attained. Let us summarize here our contribution and some of our
    expectations.
    \begin{itemize}
      \item The case $\delta=8$ is treated in \cite[Theorem
        11]{kim}, when $r=4$. For $r=5$, this is mentioned in
        \cite[Observation 14]{kim}.
      \item For $\delta=14$, it is known that a generic cubic carries
        an Ulrich bundle of rank 2, or equivalently, that it is
        Pfaffian in the classical sense. An analysis of
        $\cC_{14}$ is worked out in \cite{bolognesi-russo-stagliano}.
      \item For $\delta=18$, the divisor  $\cC_{18}$ is studied
        in \cite{A--H--T--VA}. Our contribution is summarized in
        Theorem \ref{18-theorem}, that clarifies the connections
        between various conditions, notably the presence of Ulrich
        bundles of rank 3, rank-2 instanton bundles of charge 1,
        certain surfaces in $X$, or being linear sections of Coble
        cubics. This reproves and expands \cite[Theorem
        1.3]{truong-yen}, in a geometric and computer-free fashion.
      \item For $\delta=20$, our contribution is Theorem
        \ref{20-theorem}, where we show that a general cubic fourfold $X$
        in $\hat{\cC}_{20}$ carries an Ulrich
        bundles of rank 4 and rank-2 instanton bundles of charge 2,
        and is also Steiner-Pfaffian of type $\Omega_{\bP}(1)^{\oplus
          2}$, while we already observed that any of these conditions
        forces $X$ to lie in $\hat{\cC}_{20}$. We need the aid of a computer here.

        However we do not know whether \textit{all} smooth
        Steiner-Pfaffian cubic fourfolds of type
        $\Omega_{\bP}(1)^{\oplus 2}$ carry Ulrich bundles of rank 4 or
        whether carrying an Ulrich bundles of rank 4 forces the
        presence of a rank-2 instanton of charge 2, or even of a
        rank-3 instanton of charge 1.
        Also, we do not know if a general cubic fourfold $X$
        in $\hat{\cC}_{20}$ carries an Ulrich bundle of rank 5.
        However, we expect these properties to hold true, in view of 
        speculations about Brill-Noether loci similar to those explained
        in the next item.
      \item For rank $r \in \{4,5\}$ and $\delta$ in the above table, different from $8,20$, we do not know if
        a general cubic fourfold $X$ in $\hat{\cC}_\delta$ carries an Ulrich
        bundle of rank $r$. At least for $\delta \in \{8,14,20\}$ we
        do think that this should be the 
        case. Our expectation is based on the following heuristic
        observation. It is known by \cite{faenzi-kim} and expected by
        \cite{manivel-ulrich}  that a general cubic fourfold $X$
        in $\hat{\cC}_\delta$ supports families of Ulrich bundles of rank
        $6$ and $9$, of dimension $26$ and $56$, respectively. Then,
        once fixed an Ulrich bundle $\cU_0$ of rank $r_0 \in \{2, 3,
        5\}$ whose existence
        we proved, the Brill-Noether locus of Ulrich bundles $\cU$ or
        rank $r \in \{6,9\}$
        having a non-zero map $\psi$ toward $\cU_0$ is expected to be
        non-empty by dimension reasons (we do not how to prove that this expectation holds true). Then $\ker(\psi)$ will be an
        Ulrich bundle of rank $r-r_0$, and choosing conveniently $r$
        and $r_0$ should fill the above table for $\delta \in \{8,14,20\}$.
        Anyway, the range $\delta \ge 26$ remains uncharted.
      \end{itemize}
  \end{remark}

\begin{proof}[Proof of Proposition \ref{list of deltas}]
First, note that for a sufficiently general choice of $r-1$ global sections of $\cU$,
the degeneracy locus of these sections is a smooth connected
projective surface $Y \subset X$ of class $c_2(\cU)$.
It is proved in \cite{kim} that the following relations hold
\begin{align}
  \nonumber c_1(\cU) & = r h, &&
  c_2(\cU)   =  Y, \\
  \label{c3 and c4} c_3(\cU) & = \frac{1}{6} r(r+1)(r-2)h^3, &&
  c_4(\cU)  = \frac{1}{12}(-r(r^3+4 r-3)h^4+6Y^2).
\end{align}

The point is that the deformation theory of simple Ulrich
bundles takes place within the Kuznetsov category of $X$.
Namely, based on the K3 structure of $\Ku(X)$ and relying on the theory developed in
\cite{kuznetsov-manivel-markushevich,kuznetsov-markushevich,manypeople},
it was observed in \cite{faenzi-kim} that, if $\cU$ is a simple Ulrich
bundle on $X$, then $\cU$ has a smooth universal deformation space, which
carries a holomorphic symplectic structure.
The outcome of this is that the moduli space of simple sheaves having Chern character
$\bv=\ch(\cU)$  has dimension
\begin{equation}
  \label{dim-moduli}
  m = 2 - \chi(\cU^\vee \otimes \cU) = 2+\frac{r^2(3r^2-2r+3)}{4}-Y^2,
\end{equation}
which is what we want for the first statement.
Computing the discriminant $\delta$ of the sublattice of $H^{2,2}(X,\bZ)$
generated by $h^2$ and $Y$, one finds \eqref{eq:delta} and, if $3$ does
not divide $r$, then $\delta \ne 0$ and $X$ belongs to $\hat{\cC}_\delta$.

\medskip

Next,  observe that, if $\uc(X) <
6$, then $X$ carries a simple Ulrich bundle of rank $ r
\in  \{2,3,4,5\}$.
Indeed, by definition $X$ carries an Ulrich bundle of rank $s \le 5$. Then,
since any sheaf in the Jordan-Hölder filtration of  $\cU$ is again
Ulrich, $X$ supports a stable (hence simple) Ulrich bundle of rank $ r
\in  \{2,3,4,5\}$.

Now we look at the possible values of $Y^2$ under the constraints that
 the dimension of the moduli space of simple sheaves
appearing in \eqref{dim-moduli} is non-negative and that the
discriminant $\delta$ expressed in \eqref{eq:delta} is at least 8.
This, together with the parity condition given by the fact that
$\int_Xc_4(\cU) \in \bZ$, with $c_4(\cU)$ given in \eqref{c3 and c4},
gives only finitely many choices for $Y^2$ for each value of $r \in \{2,3,4,5\}$.
Plugging these choices back into \eqref{dim-moduli} and \eqref{eq:delta}
returns precisely the values of $\delta$ and $m$ appearing in the statement.
\end{proof}

\subsection{Ulrich bundles and instantons}

Again, $X$ is a smooth cubic fourfold. We assume here that $X$
supports a rank-$s$ instanton bundle $\cE$ with charge
$\quantum$ and check that this implies the existence of an Ulrich
bundle of rank $r=s+k$.

\begin{lemma}
\label{lG}
Let $\cE$ be a rank-$s$ instanton bundle with charge
$\quantum$. Then there is a vector bundle $\cG$ of rank $r=s+k$, with $h^0(\cG(-1))=0$, fitting into an exact sequence
\begin{equation}
\label{seqExtension}
0\longrightarrow\cE(-1)\longrightarrow\cG(-1)\longrightarrow\cO_X^{\oplus
k}\longrightarrow0.
\end{equation}
Moreover, $\cG (-1)$ lies in $\Ku(X)$ and $\cG$ has the Chern character of an
Ulrich bundle. Finally, the sheaf $\cE$ is simple if and only if $\cG$ is simple.
\end{lemma}
\begin{proof}
Since $\dim \Ext_X^1(\cO_X,\cE(-1))=h^1(\cE(-1))=\quantum$, it follows
that we can consider $k$ independent sections
$u_1,\ldots,u_k\in\Ext_X^1(\cO_X,\cE(-1))$.
Thus $u:=(u_1,\ldots,u_k)\in\Ext_X^1(\cO_X^{\oplus k},\cE(-1))$
induces \eqref{seqExtension}.
Trivially $\cG$ is a vector bundle of rank $r=s+k$.
Moreover, by
applying $\Hom_X(\cO_X^{\oplus k},\cdot)$ to \eqref{seqExtension}, the choice of $u$ implies that the connecting map
$$
\Hom_X(\cO_X^{\oplus k},\cO_X^{\oplus
  k})\mapright\partial \Ext_X^1(\cO_X,\cE(-1))^{\oplus k} 
$$
is a isomorphism, as it sends a square matrix $(a_{i,j})$ of size $k$ to 
the vector $(\sum_{1 \le i \le k} a_{i,j} u_i)_{1 \le j \le k}$ of $ \Ext_X^1(\cO_X,\cE(-1))^{\oplus k} $.
In particular $\Hom_X(\cO_X^{\oplus k},\cG(-1))=0$, so $h^0(\cG(-1))=0$. 
Therefore, taking cohomology of \eqref{seqExtension} and using 
the instanton
condition of Definition \ref{dInstanton}, we see that this actually
implies $H^*(\cG(-1))=0$, and moreover  $H^*(\cG(-2))=H^*(\cG(-3))=0$, so that $\cG (-1)$ lies in
$\Ku(X)$.
Furthermore, $\chi(\cG(-4))=\chi(\cE(-4))+k\chi(\cO_X(-3))=0$, so
$\cG$ has the Chern character of an Ulrich bundle.
\medskip

Let us denote by $\zeta$ the inclusion $\cE \to \cG$ so that \eqref{seqExtension} reads
\begin{equation} \label{zeta}
    0 \to \cE(-1) \xrightarrow{\zeta} \cG(-1) \to \cO_X^{\oplus k} \to 0.
\end{equation}
Since we proved $H^*(\cG(-1))=0$, by applying $\Hom_X(\cdot,\cG(-1))$ to \eqref{zeta} we obtain isomorphisms
\begin{equation*}
  \zeta^*_p : \Ext^p_X(\cG,\cG)\simeq \Ext^p_X(\cE,\cG), \qquad \mbox{for all $p \in \bN$.}  
\end{equation*}

Applying $\Hom_X(\cE(-1),\cdot)$ to \eqref{zeta}, for all $p \in \bN$, we get a map
\[
\zeta_*^p : \Ext^p_X(\cE,\cE) \to \Ext^p_X(\cE,\cG)
\]
Now, for all $p \in \{0,\ldots,4\}$, we note that
$\Ext^p_X(\cE(-1),\cO_X)^\vee\cong H^{4-p}(\cE(-4))$ and recall that
$h^{4-p}(\cE(-4))=0$ for $p \ne 1$, while $h^{3}(\cE(-4))=k$.
We deduce that there is an isomorphism
\[
  (\zeta_*^0)^{-1} \circ \zeta_0^* : \Hom_X(\cG,\cG) \simeq \Hom_X(\cE,\cG) \simeq \Hom_X(\cE,\cE),
\]
So the last statement is proved. 
\end{proof}

We will again use the maps $\zeta_*^p$ and $\zeta^*_p$ defined in the previous proof.
The reader should be warned that the notation  $\zeta_*^p$ and $\zeta^*_p$ does not keep track of the functor we apply to the map $\zeta$, as only the cohomological degree and the covariant or contravariant nature of the functor is recorded.

\begin{definition}
    Let $\cE$ be an instanton sheaf on $X$ and let $\cG(-1)$ be the
  extension of $\cE(-1)$ defined by the previous lemma.
  Then we call $\cG$ the
  \textit{acyclic extension} of $\cE$. 
\end{definition}

Note that this terminology is a bit abusive, since $\cG(-1)$, rather than $\cG$, is actually acyclic, as it lies in $\Ku(X)$ by the previous lemma.

\begin{lemma} \label{at-least-1}
Let $\cE$ be an instanton sheaf and $\cG$ its acyclic extension. Then there is a surjection
  \[
    \Ext_X^2(\cE,\cE) \xrightarrow{\zeta^2_*\circ(\zeta^*_2)^{-1}}\Ext_X^2(\cG,\cG) \simeq \Hom_X(\cG,\cG)^\vee.
  \]
\end{lemma}

\begin{proof}
    The last isomorphism comes from the fact that $\cG (-1)$ lies in
  the K3 category $\Ku(X)$. 
  To get the surjection, we apply $\Hom_X(\cdot,\cE(-1))$ to \eqref{zeta} and use that $h^2(\cE(-1))=h^3(\cE(-1))=0$. This gives an isomorphism
  \[
  \zeta^*_2 : \Ext^2_X(\cG,\cE) \to \Ext^2_X(\cE,\cE).
  \]
  Then we apply $\Hom_X(\cG(-1),\cdot)$ to \eqref{zeta} and use $\Ext^2_X(\cG,\cO_X(1))=0$, which follows from the previous lemma by Serre duality.
\end{proof}

When $\cE$ is a simple instanton sheaf, then we have a surjection
\begin{equation}
\label{one-d}    
    \Ext_X^2(\cE,\cE) \xrightarrow{\zeta^2_*\circ(\zeta^*_2)^{-1}} \Ext^2_X(\cG,\cG) \simeq \bk.
  \end{equation}

  Hence $\Ext_X^2(\cE,\cE)$ is always at least
  one-dimensional.
  Recall that, in our terminology, $\cE$ is called unobstructed if $\Ext_X^2(\cE,\cE)$ is precisely one-dimensional.

\begin{proposition} \label{deforms to Ulrich}
  Let $\cE$ be a simple unobstructed instanton on $X$ and let $\cG$ be the
  acyclic extension of $\cE$, given by Lemma \ref{lG}.
  Then $\cG$ deforms flatly to an Ulrich bundle on $X$.
\end{proposition}

\begin{proof}
  By Lemma \ref{lG}, $\cG (-1)$ lies in $\Ku(X)$ and $\cG$ has the
  Chern character of an Ulrich bundle. Hence the first part of the
  proof of Proposition \ref{list of deltas} applies to $\cG$ in place
  of $\cU$ and shows that $\cG$ deforms flatly in a smooth family of sheaves
  of dimension
  \begin{equation}
    \label{mG}
  m_\cG = 2 - \chi(\cG ^\vee \otimes \cG) = 2+\frac{r^2(3r^2-2r+3)}4-Y^2,
  \end{equation}
  where $r=\rk(\cG) = s+k$ and $Y^2=c_2(\cG)$,
  $s$ and $k$ being the rank and the charge of $\cE$.

  \medskip
  
  Let $\cG'$ be a general sheaf, lying in a smooth connected
  neighborhood of $\cG$
  inside the moduli space of simple sheaves having Chern character
  $\ch(\cG)$. To prove that $\cG'$ is an Ulrich bundle, it suffices to
  show that $\Hom(\cG'(-1),\cO_X)=0$.
  Indeed first note that $\Hom(\cG'(-1),\cO_X)$ is dual to
  $H^4(\cG'(-4))$ and that $\cG'(-1)$ lies in $\Ku(X)$.
  Second, observe that  $H^i(\cG(-4))=0$
  for $i=0,1,2$ so that $H^i(\cG'(-4))=0$ for $i=0,1,2$ by
  semicontinuity. Finally, since $\chi(\cG'(-4))=0$, we have that $H^4(\cG'(-4))=0$
  implies $H^3(\cG'(-4))=0$ and $\cG'$ is Ulrich.

  \medskip
  
  To check that $\Hom(\cG'(-1),\cO_X)=0$, we assume the contrary and
  seek a contradiction. Say that, for $\cG'$ general in a neighborhood
  of $\cG$, the space
  $\Hom(\cG'(-1),\cO_X)$ is of dimension $\ell$ for some $\ell > 0$.
  Again by semicontinuity we have $\ell \le k$. Then, the coevaluation
  of sections $g : \cG'(-1) \to \cO_X^{\oplus \ell}$ is
  surjective, for this map specializes to the composition $\cG'(-1) \to
  \cO_X^{\oplus k} \to \cO_X^{\oplus \ell}$ for some choice of a
  surjection $\cO_X^{\oplus k} \to \cO_X^{\oplus \ell}$. Therefore
  $\ker(g)$ is a vector bundle $\cK'(-1)$, where $\cK'$ is actually an
  instanton on $X$, of rank $r-\ell$ and charge $\ell$. Then $\cK'$ is a flat
  deformation of $\cK$, a vector bundle fitting into
  \begin{equation}
    \label{K'}
    0 \to \cE(-1) \to \cK(-1) \to \cO_X^{\oplus (k-\ell)} \to 0.    
  \end{equation}

  Again $\cK$ is simple by Lemma \ref{lG}, because the acyclic
  extension of $\cK$ is still $\cG$, so that the moduli space of
  simple sheaves with Chern character $\ch(\cK')$
  has dimension $m_{\cK'}$, with
  \begin{equation}
    \label{<mK}
    m_{\cK'} \le \dim (\Ext^1_X(\cK',\cK')) \le \dim (\Ext^1_X(\cK,\cK)),    
  \end{equation}
  where the second inequality follows again by semicontinuity.
  \medskip
  
  Under our assumption, we may define a
  rational map by sending  $\cG'$ to $\cK'$, defined in the
  neighborhood of $\cG$ consisting of simple sheaves $\cG'$ in $\Ku(X)$
  such that $\Hom(\cG'(-1),\cO_X)$ is of dimension $\ell$. Note that this map is generically
  injective, for $\cG'$ can be recovered from $\cK'$ as the acyclic
  extension of $\cK'$. Therefore we must have
  \begin{equation}
    \label{mG<mK'}
  m_\cG \le m_{\cK'}.  
  \end{equation}
  
  Further, we observe that, under the assumption that $\cE$ is unobstructed, we have
  \begin{equation}
    \label{ExtE}
   \dim (\Ext^1_X(\cE,\cE)) = 2+\frac {r^2(3r^2-2r+3)}4 -Y^2-k^2.    
  \end{equation}
Indeed, this follows easily from Riemann-Roch, together with the fact that $\Ext_X^p(\cE,\cE)$ is
zero for $p > 2$ and one-dimensional for $p=0,2$.
\medskip

To arrive to a contradiction, we give an upper bound 
$\dim (\Ext^1_X(\cK,\cK))$.
With the same argument as in the proof of Lemma \ref{lG}, we apply $\Hom_X(\cdot,\cK(-1))$
  and $\Hom_X(\cdot,\cE(-1))$ to \eqref{K'}. Since ${\cK}'$ is an instanton of charge $l$, 
  we get:
  \begin{align} \label{K<E}
   \dim (\Ext^1_X(\cK,\cK)) &\le  \dim
    (\Ext^1_X(\cE,\cE))+(k-\ell)(h^1(\cK(-1))+h^3(\cE(-4)) =\\
                            \nonumber &= \dim (\Ext^1_X(\cE,\cE)) + k^2-\ell^2.
  \end{align}
  Therefore, putting together \eqref{mG}, \eqref{<mK}, \eqref{K<E},
  \eqref{mG<mK'} and \eqref{ExtE}, we get:
  \begin{align*}
    2+\frac{r^2(3r^2-2r+3)}4 -Y^2 = m_\cG &\le m_{\cK'} \le \dim
                                            (\Ext_X^1(\cK,\cK)) \le \\
                                          & \le \dim (\Ext^1_X(\cE,\cE)) + k^2-\ell^2 \le \\
                                          & \le 2+\frac {r^2(3r^2-2r+3)}4 -Y^2-\ell^2 
  \end{align*}
  However, this just gives  $\ell^2 \le 0$, which is against our
  assumption that $\ell > 0$.
  This contradiction achieves the proof of the proposition.
\end{proof}

Putting together the previous propositions completes the  proof of Theorem \ref{theorem ulrich}. Indeed, the statements about the dimension and symplectic nature of $\fU_X(r,a)$ are given in Proposition \ref{list of deltas}, as well as the details avout $r,\delta$ and $m$. Also, given an instanton bundle $\cE$, the acyclic extension $\cG$ of $\cE$ is a deformation of $\cE \oplus \cO_X(1)^{\oplus k}$ which in turn deforms to an Ulrich bundle by Proposition \ref{deforms to Ulrich}.

\subsection{Unobstructed instantons}

\label{section:unobstructed}

To conclude this section, we show that, on a smooth cubic fourfold, unobstructed instantons, which we defined as simple instanton sheaves such that $\dim(\Ext^2_X(\cE,\cE))$ equals one (which is the minimal possible value of such dimension) define smooth points of the moduli space of simple sheaves on $X$.
This justifies the adjective \textit{unobstructed} and allows to compute the dimension of such moduli space in many situations.
One should notice the difference for instance with respect to cubic threefolds $Y$, where unobstructed instantons $\cE$ are typically defined as having $\Ext^2_Y(\cE,\cE)=0$ -- we refer for instance to \cite{faenzi-instanton, faenzi-comaschi} for a construction of stable instanton bundles $\cE$ of any admissible charge and rank on a cubic threefold $Y$ satisfying  the vanishing $\Ext^2_Y(\cE,\cE)=0$.
 \medskip
 
Here we prove the main result of this subsection, which is Theorem \ref{unobstructed}.
As a first observation, notice that the assumption that $\cE$ is simple in the statement of Theorem \ref{unobstructed} is actually redundant, so that Theorem \ref{unobstructed-intro} from the introduction will be also proved.
Indeed, if $\cE$ is an unobstructed instanton sheaf, then $\Ext^2_X(\cE,\cE)$ is one-dimensional by definition, so that, if $\cG$ is the acyclic extension of $\cE$, then $\Ext^2_X(\cG,\cG)$ is also one-dimensional by Lemma \ref{at-least-1}. Hence $\cG$ is simple, because $\cG$ lies in $\Ku(X)$. Thus, by Lemma \ref{lG}, $\cE$ is also simple.

  \begin{proof}[Proof of Theorem \ref{unobstructed}]
      Following \cite{buchweitz-flenner:atiyah, buchweitz-flenner:semiregularity}, see also \cite{kuznetsov-markushevich}, one introduces the map
      \[
      \sigma^\cE = \sum_{q \ge 0} \sigma^\cE _q : \Ext^2_X(\cE,\cE) \to \bigoplus_{q \ge 0} H^{q+2}(X,\Omega_X^q).
      \]
      Here $\sigma^\cE_q$ is defined by first cupping a class $\xi \in \Ext^2_X(\cE,\cE)$ with the $q$-th exterior power of the Atiyah class of $\cE$,  which we denote by $\wedge^q \At(\cE) \in \Ext_X^q(\cE,\cE \otimes \Omega^q_X)$ and then applying the trace map $\tr : \Ext^{q+2}_X(\cE,\cE \otimes \Omega^q_X) \to H^{q+2}(X,\Omega_X^q)$. More precisely, for a smooth cubic fourfold $X$, the relevant part of $\sigma ^cE$ lands into $H^3(X,\Omega_X)$, which is identified with $\bk$, once we choose a generator of $H^{1,3}(X)$.
      To write $\sigma ^\cE$ more explicitly, we denote by $\Upsilon$ the Yoneda product of any triple of sheaves $\cA, \cB, \cC$, so that we have, for any $p,q \in \bN$:
      \[
      \Upsilon : \Ext^p_X(\cA,\cB) \otimes \Ext^q_X(\cB,\cC) \to \Ext^{p+q}_X(\cA,\cC).
      \]
      Incidentally, we recall that $\Upsilon$ is functorial in all its arguments with respect to morphisms of sheaves. For instance, let us spell out one incarnation of functoriality with respect to a morphism $\theta : \cB \to \cB'$. We have a diagram
      \[
      \xymatrix@-2ex{
      \Ext^p_X(\cA,\cB) \otimes \Ext^q_X(\cB,\cC) \ar_{\theta^p_*}[d] \ar^-\Upsilon[drrr] & \\ 
      \Ext^p_X(\cA,\cB') \otimes \Ext^q_X(\cB,\cC)  &&& \Ext^{p+q}_X(\cA,\cC) \\
      \Ext^p_X(\cA,\cB') \otimes \Ext^q_X(\cB',\cC) \ar^{\theta_q^*}[u] \ar_-\Upsilon[urrr]
      }
      \]
    Given classes $\alpha \in \Ext^p_X(\cA,\cB)$, $\beta \in \Ext^q_X(\cB,\cC)$, 
    $\alpha' \in \Ext^p_X(\cA,\cB')$, $\beta' \in \Ext^q_X(\cB',\cC)$ satisfying 
    \[\theta_*^p(\alpha)=\alpha', \qquad \theta^*_q(\beta')=\beta,
    \]
    the functoriality of $\Upsilon$ is expressed by the equality
    \[
    \Upsilon(\alpha\otimes \beta)=\Upsilon(\alpha'\otimes \beta').
    \]
      
      Having set this up, we can write explicitly the formula for $\sigma^\cE=\sigma_1^\cE$, which reads 
      \[\sigma^\cE(\xi) = \tr(\Upsilon(\xi \otimes \At(\cE))).\]          
    Now, according to \textit{loc. cit.}, the moduli space of simple sheaves on $X$ is smooth at $[\cE]$ if the map $\sigma$ is injective. So our goal will be to check that $\sigma^\cE_1$ is indeed injective.

    Let $\cG$ be the
  acyclic extension of $\cE$. 
  We know by the argument of \cite[Theorem 31.1]{manypeople}, which is also underlying the proof of Proposition \ref{list of deltas}, that $\sigma^\cG$ is injective at the point $[\cG]$ of the moduli space of simple sheaves on $X$.
Under our assumption, by Lemma \ref{at-least-1}, $\Ext^2_X(\cE,\cE)$ is identified with $\Ext^2_X(\cG,\cG)$ by the map $\zeta^2_*\circ(\zeta^*_2)^{-1}$. Note that both $\zeta^2_*$ and $\zeta^*_2$ are isomorphisms in this case as we checked in the proof of Lemma \ref{at-least-1} that $\zeta^*_2$ is always an isomorphism and their composition is an isomorphism by assumption.
Moreover, we have a natural injective map
  \[
  \psi : \Ext^1_X(\cG,\cG \otimes \Omega_X) \to \Ext^1_X(\cE,\cE \otimes \Omega_X)
  \]
  such that $\psi(\At(\cG))=\At(\cE)$.
  To get this map, we first apply $\Hom_X(-,\cG(-1) \otimes \Omega_X)$ to \eqref{zeta}. We recall that $h^2(\cG(-4))=h^1(\cG(-2)=h^0(\cG(-1))=0$, so that dualizing \eqref{seqTangentHypersurface} and \eqref{seqNormal} and tensoring with $\cG(-1)$
  we get $h^1(\cG(-1) \otimes \Omega_X)=0$.
  This provides an injection  
    \[
    \zeta^*_1 : \Ext^1_X(\cG,\cG \otimes \Omega_X) \hookrightarrow
    \Ext^1_X(\cE,\cG \otimes \Omega_X) 
    \]
  Then, tensoring \eqref{zeta} with $\Omega_X(1)$ and applying $\Hom_X(\cE,-)$ gives
  a map
  \[
  \zeta_*^1 : \Ext^1_X(\cE,\cE
  \otimes \Omega_X) \to \Ext^1_X(\cE,\cG \otimes \Omega_X) 
  \]
  We check that $\zeta_*^1$ is an isomorphism, since
  \[
  \Ext^q_X(\cE,\Omega_X(1))=0, \qquad \mbox{for $q \in \{0,1\}$.}
  \]
  Indeed, using Serre duality, tensoring with $\cE(-4)$, the sequences \eqref{seqTangentHypersurface} and \eqref{seqNormal}, we see that the above vanishing follows from $h^{q+1}(\cE(-3))=h^{q}(\cE(-1))=0$ for $q \in \{2,3\}$ and $h^4(\cE(-4))=0$, which in turn are given by Definition \ref{dInstanton}.
In conclusion, the map $\psi$ is $(\zeta_*^1)^{-1} \circ \zeta_1^*$. 
The fact that $\zeta_*^1(\At(\cE))=\zeta_1^*(\At(\cG))$ comes from the functoriality of the Atiyah class.

\medskip
We are in position to check that $\sigma_1^\cE$ is injective. Let $\xi \in \Ext^2_X(\cE,\cE)$ satisfy 
$\sigma_1^\cE(\xi)=0$.
Set $\eta = \zeta^2_*((\zeta^*_2)^{-1}(\xi))$,
so $\eta \in \Ext^2_X(\cG,\cG)$ vanishes if and only if $\xi$ vanishes.
We have:
\begin{align*}
    0 = \sigma_1^\cE(\xi)=& \tr(\Upsilon(\xi \otimes \At(\cE))) = \\
    =& \tr(\Upsilon(\zeta_2^*((\zeta_*^2)^{-1}(\eta)) \otimes \At(\cE))) = \\
=& \tr(\Upsilon(\zeta_2^*((\zeta_*^2)^{-1}(\eta)) \otimes (\zeta_*^1)^{-1}(\zeta_1^*(\At(\cG))))) = \\
=& \tr(\Upsilon((\zeta_*^2)^{-1}(\eta) \otimes \zeta_1^*(\At(\cG)))) = \\
=& \tr(\Upsilon(\eta \otimes \At(\cG))) = \sigma_1^\cG(\eta),
\end{align*}
where the equalities follow from the functoriality of the Yoneda pairing mentioned above with respect to the morphism $\zeta$. 
Therefore, since $\sigma_1^\cG$ is injective, we get $\eta=0$, so that $\xi=0$ and $\sigma_1^\cE$ is also injective.
\end{proof}

\section{Rank two instanton bundles on smooth cubic hypersurfaces}
\label{section:instanton of rank 2}

We discussed in the previous section the relationship between
instantons, Ulrich
complexity, and Hassett divisors, expressed by Theorem \ref{theorem ulrich}.
From now on we focus on instantons of rank 2 and analyse some
particular Hassett divisors under this point of view.
Note that if $\cE$ is an instanton of rank $2$ and charge $k$ on a smooth cubic
fourfold $X$, then $0 \le k \le 7$ by Corollary \ref{cBound}.
Also, if $Y=c_2(\cE(k h))$ is homologous to a complete intersection in $X$, then $Y\sim\lambda h^2$ where
$$
\lambda:=\frac13(3\quantum^2+7\quantum+5)\in\bZ.
$$
thanks to \eqref{Degree} with $s=k$.
Moreover $-k^2 +5k
+14=3Y^2-\lambda^2=0$.
It follows that $\quantum=7$:  in this case
$h^2Y=201$, hence it could be $Y\sim 67h^2$ inside $X$. In particular
$Y$ should be homologically equivalent to a complete intersection of
divisors of degree $1$ and $67$ inside $X$. 
Thus the non-special smooth cubic hypersurface cannot support rank
two instanton bundles with charge $\quantum\ne7$. As pointed
out in \cite[Example 7.2]{An--Cs2} there exist special smooth cubic
hypersurfaces endowed with a rank two instanton bundle with charge $0\le \quantum\le 4$.
On the other hand, the existence of smooth cubic hypersurfaces endowed
with a rank two instanton bundle with charge $5\le \quantum\le
7$ is wide open. 

Using Formula \ref{delta-formula} of Lemma  \ref{delta}, or
Proposition \ref{pSurfaceY}, one obtains the following table of
possible values of the classes $Y^2$, $h^2Y$ and their discriminant
$\delta$, or equivalently of the discriminant of $h^2c_2(\cE)$ and
$c_2(\cE)^2$, where $\cE$ is a rank-2 instanton of charge $k$ on a
smooth cubic fourfold $X$ and $Y$ is a surface appearing as zero-locus of a
general section of $\cE(k)$.

\begin{table}[H] \label{table-k-delta}
\centering
\renewcommand{\arraystretch}{1.1}
\begin{tabular}{c|c|c|c|c|c|c|c|c}
\toprule
{$\quantum$} & 0 & 1 & 2 & 3 & 4 & 5 & 6 & 7 \\
  \midrule
  \midrule
  $\delta$     & 14 & 18 & 20 & 20 & 18 & 14 & 8  & 0   \\
\midrule
  $h^2 Y$      & 5  & 15 & 31 & 53 & 81 & 115& 155& 201 \\
\midrule
  $Y^2$        & 13 & 81 & 327& 943& 2193& 4413& 8011& 13467 \\
\bottomrule
\end{tabular}
\caption{\textit{Values of $\quantum$, $\delta$, $h^2 Y$, and $Y^2$.}}
\end{table}

In what follows we will characterize $\cC_{14}$, $\cC_{18}$ and $\cC_{20}$ in terms of rank two instanton bundles with low charge. 
We start with an easy lemma.
\begin{lemma} \label{silly lemma}
  Let $\cE$ be a rank-2 vector bundle on a smooth cubic fourfold $X$
  with $c_1(\cE)=2h$. Then $\cE$ is an instanton of charge $k$ if and
  only if:
  \begin{align*}
    h^0(\cE(-1)) & = 0, &&   h^1(\cE(-1)) = k,\\
    h^1(\cE(-2)) & = 0, && h^2(\cE(-2)) = 0.
  \end{align*}
\end{lemma}

\begin{proof}
Since $\wedge^2 \cE \simeq \cO_X(2)$, we have
$\cE^\vee \simeq \cE(-2)$. Then Serre duality gives, for all $p \in
\{0,\ldots,4\}$ : 
\begin{align*}
    h^p(\cE(-3))& =h^{4-p}(\cE^\vee)=h^{4-p}(\cE(-2)),\\
    h^p(\cE(-4))& =h^{4-p}(\cE(-1)).
\end{align*}
Therefore, the conditions are sufficient to comply with Definition
\ref{dInstanton}.
\end{proof}

\subsection{General cubic hypersurfaces in
  $\cC_{14}$.} \label{14:section} 
Recall that a smooth cubic hypersurface $X\subseteq\p5$ supports a
rank two Ulrich bundle if and only if it is linear Pfaffian, if and
only if it contains a Del Pezzo surface of degree 5 (e.g. see \cite[Proposition 9.2]{Bea}).

In view of what we will prove in the two next subsections we close this subsection with another different characterization of the general Pfaffian cubic hypersurface in $\p5$. 

\begin{proposition}
\label{pEquivalence14}
A smooth cubic fourfold $X$ supports a rank two Ulrich bundle $\cE$ if and only if there is a smooth, connected surface $Y\subseteq X$ not contained in any quadric and such that $\deg(Y)=14$, $q(Y)=0$, $p_g(Y)=6$, $\omega_Y\cong\cO_Y(1)$.
\end{proposition}
\begin{proof}
If $\cE$ exists, then it is globally generated, hence the zero locus $Y$ of a section of $\cE(1)$ is a smooth surface, at least in characteristic zero. If $\bk$ has positive characteristic, we may argue anyway that $Y$ is smooth, for already a section of $\cE$ vanishes on a smooth quintic Del Pezzo surface, so since $h$ is very ample also a section of $\cE(1)$ vanishes on a smooth surface. Thanks to Proposition \ref{pSurfaceY} we know that $Y$ is connected, non-degenerate, linearly normal, $\omega_Y\cong\cO_Y\otimes\cO_{\p5}(1)$, $K_Y^2=14$ and $q(Y)=0$: in particular $\varrho^Y_1$ is an isomorphism, hence $p_g(Y)=6$. 
Moreover, we also know that $\varrho^Y_2$ has maximal rank, again by Proposition \ref{pSurfaceY}. Since
\begin{equation}
  \label{p2=21}
h^0(\cO_{\p5}(2))=h^0(\cO_{X}(2))=h^0(\cO_Y(2))=h^0(\omega_Y^2)=21,
\end{equation}
by \cite[Corollary VII.5.3]{B--H--P--VV}, we get
$h^0(\cI_{Y\vert\p5}(2))=h^0(\cI_{Y\vert X}(2))=0$, i.e. $Y$ is contained in no quadric.

Conversely, assume the existence of a surface as in the statement. Theorem \ref{tSerre} implies the existence of a unique rank two vector bundle $\cA$ fitting into an exact sequence 
$$
0\longrightarrow\cO_X\longrightarrow\cA\longrightarrow\cI_{Y\vert X}(4)\longrightarrow0.
$$
Put $\cE:=\cA(-1)$, so $c_1(\cE)=2h$. Let us check that $\cE$ is
an instanton of charge $0$, keeping in mind Lemma \ref{silly lemma}.
Using
 \eqref{seqChain}, we get
\begin{gather*}
h^p(\cE(-1))=h^p(\cI_{Y\vert X}(2)),\qquad h^q(\cE(-2))=h^q(\cI_{Y\vert X}(1)),
\end{gather*}
for all $p \ge 0$  and $0 \le q \le 2$.
Since $Y$ is contained in no quadric, $h^0(\cI_{Y\vert X}(2))=0$,
which implies $h^0(\cI_{Y\vert X}(1))=0$, so $\varrho^Y_1$ is
injective, hence an isomorphism, as we are assuming $h^0(\cO_Y(1))=h^0(\omega_Y)=p_q(Y)=6$. 
Then we get $h^1(\cI_{Y\vert X}(1))=0$.
Also, using again $h^0(\cI_{Y\vert X}(2))=0$ and \eqref{p2=21}, we see
that $\varrho^Y_2$ is an isomorphism, so $h^1(\cI_{Y\vert X}(2))=0$.
Finally, we have
\[h^2(\cI_{Y\vert X}(1))=h^1(\cO_Y(1))=h^1(\omega_Y)=q(Y)=0.\]
In conclusion, $\cE$ is
an instanton of charge $0$ by Lemma \ref{silly lemma}.
\end{proof}

\begin{remark}
\label{rCatanese}
The existence of surfaces as in the statement is well-known.
Indeed, there is a self-dual resolution
$$
0\longrightarrow\cO_{\p5}(-7)\longrightarrow\cO_{\p5}(-4)^{\oplus7}\longrightarrow\cO_{\p5}(-3)^{\oplus7}\longrightarrow\cI_{Y\vert\p5}\longrightarrow0:
$$
see \cite[Section 4]{Cat}. 
\end{remark}

\begin{remark}
It is interesting to notice that the surfaces in the statement of
Proposition \ref{pEquivalence14} also characterize Pfaffian quartic
hypersurfaces in $\p5$,  see \S \ref{pQuartic0}.
\end{remark}

\subsection{Cubic hypersurfaces in $\cC_{18}$.}

\label{18-section}

Let $X$ be a smooth cubic fourfold. The goal of this subsection is to
establish the logical implications among the following
conditions.

  \begin{enumerate}[label=\roman*)]
  \item \label{18-i} $X$ supports an unobstructed rank-2 instanton of charge 1;
  \item \label{18-ii} $X$ supports a rank-2 instanton of charge 1;
  \item \label{18-iii} $X$ is Steiner-Pfaffian of type $\cO_{\bP^5}^3 \oplus
    \Omega_{\bP^5}(1)$;
  \item \label{18-iv} $X$ is a linear section of a Coble cubic;
\item \label{18-v} $X$ contains a connected non-degenerate LCI
  surface $Y$ with
  \[\omega_Y\simeq \cO_Y(-1), \qquad h^1(\cO_Y)=0, \qquad
    h^0(\cO_Y(1))=7.
    \]
\item \label{18-vi} $X$ contains a smooth connected
surface $Y$, contained in no quadric, with
\[\deg(Y)=15, \quad q(Y)=0, \quad p_g(Y)=6, \quad \omega_Y\cong\cO_Y(1).\]
    \item \label{18-vii} $X$ supports an Ulrich bundle of rank 3;
    \item \label{18-viii} $X$ lies in $\hat{\cC}_{18}$.
    \end{enumerate}

As a matter of terminology, recall that a subvariety of $\bP^n$ is
non-degenerate if
it is contained in no hyperplane. 
The definitions about Pfaffian-Steiner representations of a given
type, instanton and Ulrich bundles have been given in \S
\ref{section:introduction}. We say that a stable bundle is
unobstructed if it is a smooth point of the corresponding moduli
space.
LCI here stands for \textit{locally
  complete intersection}. A \textit{canonical} surface in $\bP^n$ is a
surface $Y$ embedded by the complete linear system $|\omega_Y|$.
Here, a \textit{Coble cubic} is a 
hypersurface $\Gamma$ arising from alternating 3-forms $\omega$ in 9
variables. We refer to \cite{GSW, gs} and to \S \ref{18-proof} for
more details.
Under a suitable generality condition on the form $\omega$, a smooth
hypersurface of this kind appears as the unique cubic hypersurface in $\bP^8$ whose
singular locus is the Jacobian of a given curve of genus two,
naturally associated with $\omega$, embedded by
$|3\Theta|$, according to Coble's work \cite{point-sets-III}.
Our main result here is the following (we assume that $\bk$ has characteristic zero).

\begin{theorem} \label{18-theorem}
  For any smooth cubic fourfold, we have
      \[
        \ref{18-vii}         \Leftarrow        \ref{18-i} \Rightarrow \ref{18-ii} \Leftrightarrow
        \ref{18-iii} \Leftrightarrow    \ref{18-iv} \Leftrightarrow   \ref{18-v} \Leftrightarrow     \ref{18-vi}
      \]
      Any condition implies $\ref{18-viii}$. Conversely, if
      $X$ is generic in  
      $ \hat \cC_{18},$ 
      then all these conditions hold and 
      \[
        \dim (\fU_X(3))=2, \qquad \dim (\fI_X(2,1))=1.
      \]
\end{theorem}

The proof of the above theorem will occupy the rest of the subsection. The proof will
be summarized in \S \ref{18-proof}.
This will prove a stronger version of Theorem \ref{C18-intro} from the introduction. Note that, in Theorem \ref{C18-intro}, it is mentioned that the Ulrich bundles $\cU$ under consideration are stable. However, this is obvious if $X$ represents a point of $\cC_{18} \setminus \cC_{14}$, since $\cU$ can only be destabilized by an Ulrich bundle of rank $2$, which would force $X$ to lie in $\hat \cC_{14}$.

\subsubsection{Rank-two instantons of charge one}

Let us focus on the case of rank two instanton bundles with charge $\quantum=1$. 
First we consider the general isomorphic projection in $\p5$  of a smooth Del Pezzo surface of degree $6$ in $\p6$. Thanks to \cite[Proposition 2.1]{KapG} we know that there is a smooth cubic hypersurface $X\subseteq\p5$ containing it. 

\begin{proposition} \label{pInstanton1}
  A smooth cubic fourfold $X$ supports a rank-2 instanton bundle of
  charge 1 if and only if $X$ contains a connected non-degenerate LCI
  surface $Y$ with
  \[\omega_Y\simeq \cO_Y(-1), \qquad h^1(\cO_Y)=0, \qquad
    h^0(\cO_Y(1))=7.
    \]
\end{proposition}

\begin{proof}
  Assume that $X$ supports a rank-2 instanton bundle $\cE$ of
  charge 1. Then $\cE(1)$ is regular, so $h^1(\cE)=0$.
  Hence \eqref{ChiE} with $t=0$ yields $h^0(\cE)=3$ and we know from
  Remark \ref{rFirstSection}, Proposition \ref{pSurfaceY} with $d=3$
  and $s=0$ and  \eqref{Normal1} that $X$ contains a surface $Y$ as
  desired.

  Conversely, assume that $X$ contains a connected non-degenerate LCI
  surface $Y$ with $h^1(\cO_Y)=0$ and $h^0(\cO_Y(1))=7$. Then
  $h^p(\cI_{Y|X})=0$ for $p=0,1,2$ and $h^0(\cI_{Y|X}(1))=0$, while
  $h^1(\cI_{Y|X}(1))=1$.
  Thus Theorem \ref{tSerre} implies the existence of a unique rank two
  vector bundle $\cE$ with $c_1(\cE)=2h$, fitting into
  \eqref{seqIdeal} with $d=3$ and $s=0$.
  Using the sequence \eqref{seqIdeal}, we compute:
  \begin{align*}
    h^0(\cE(-1)) & = h^0(\cI_{Y|X}(1))=0, \\
    h^1(\cE(-1)) & =  h^1(\cI_{Y|X}(1))= 1,\\
    h^k(\cE(-2)) & =  h^k(\cI_{Y|X})= 0, && \forall k \in \{0,1,2\}.
  \end{align*}
  Therefore, Serre duality gives
  \[
    h^p(\cE(-3))=h^{4-p}(\cE^\vee)=h^{4-p}(\cE(-2))=0, \qquad \forall
    p \in \{2,3,4\}.
  \]
  In conclusion, $\cE$ is an instanton bundle of charge $1$.
\end{proof}

\begin{remark}
Though the zero-locus $Y\subseteq X$ of a section of $\cE$ is a
locally complete intersection subscheme of pure dimension $2$ we are
unable to prove that $Y$ is smooth or even integral.
However, if $Y$ is integral and normal, then $Y$ is a Del Pezzo
surface of degree 6, anticanonically embedded and projected to $\bP^5$ from a point outside $Y$.
\end{remark}

\begin{versionb}
  \begin{proposition}
\label{pInstanton1}
The following assertions hold.

\begin{enumerate}
\item If $X$ supports a rank two instanton bundle $\cE$ with charge $\quantum=1$, then $h^0(\cE)=3$, $h^1(\cE)=0$ and the general section of $H^0(\cE)$ vanishes on a connected, non-degenerate, locally complete intersection surface $Y$ such that $\deg(Y)=6$, $\omega_Y\cong\cO_Y(-1)$ and $h^0(\cO_Y(1))=7$.
\item If $X$ contains the isomorphic projection $Y\subseteq\p5$  of a smooth Del Pezzo surface of degree $6$ in $\p6$, then $X$ supports a rank two instanton bundle $\cE$ with charge $\quantum=1$ and the general section of $H^0(\cE)$ vanishes a surface with the same properties as $Y$.
\end{enumerate}
\end{proposition}
\begin{proof}
Recall that $\cE(1)$ is regular, hence  $h^1(\cE)=0$. Thus \eqref{ChiE} with $t=0$ yields $h^0(\cE)=3$ and we know from Remark \ref{rFirstSection}, Proposition \ref{pSurfaceY} with $d=3$ and $s=0$ and  \eqref{Normal1} that $Y$ exists and it is as in assertion (1) above.

We prove assertion (2) below. Since $\omega_Y\cong\cO_Y(-1)$, it
follows that $\det(\cN_{Y\vert X})\cong\cO_Y(2)$.
Thus Theorem \ref{tSerre} implies the existence of a unique rank two vector bundle $\cE$ with $c_1(\cE)=2h$, fitting into \eqref{seqIdeal} with $d=3$ and $s=0$. Thus,
$$
h^i(\cE(-(i+1))=h^i(\cI_{Y\vert X}(1-i),\qquad h^1(\cE(-1)=h^1(\cI_{Y\vert X}(1)),
$$
where $0\le i\le 2$. The cohomologies of the twists of \eqref{seqChain} with $S:=Y$, $Z:=X$ and $P:=\p5$ and \eqref{seqStandard} with $X:=\p5$ imply $h^2(\cI_{Y\vert X}(-1))=h^1(\cO_Y(-1))=0$, thanks to the Kodaira vanishing theorem, 
$h^0(\cI_{Y\vert X}(1))=h^0(\cI_{Y\vert\p5}(1))=0$, because $Y$ is non-degenerate by construction,
$$
h^1(\cI_{Y\vert X}(1))=h^1(\cI_{Y\vert\p5}(1))=h^0(\cO_Y(1))-h^0(\cO_{\p5}(1))=1,
$$
because $Y$ is the projection of a non-degenerate and linearly normal surface in $\p6$, and
$$
h^1(\cI_{Y\vert X})=h^1(\cI_{Y\vert\p5})=h^0(\cO_Y)-h^0(\cO_{\p5})+h^0(\cI_{Y\vert X})=0.
$$
Thus, $h^0(\cE(-1))=h^1(\cE(-2))=h^2(\cE(-2))=0$ and $h^1(\cE(-1))=1$, hence $\cE$ is a rank two $h$-instanton bundle on $X$ with charge $\quantum=1$, thanks to \cite[Proposition 6.7]{An--Cs1}.

Arguing as in the proof of assertion (1) we know that the general section of $H^0(\cE)$ vanishes on a connected, non-degenerate, locally complete intersection surface $Y'$ such that $\deg(Y')=6$ and $\omega_{Y'}\cong\cO_{Y'}\otimes\cO_{\p5}(-1)$. Since $Y$ is smooth by definition, the same is true for $Y'$, hence $Y'\subseteq\p5$ is the  isomorphic projection of a smooth Del Pezzo surface of degree $6$ in $\p6$. Thus also assertion (2) is proved.
\end{proof}
\end{versionb}

In the following corollary we deal with the moduli space of instanton with charge $\quantum=1$.

\begin{corollary}
\label{cCharge1}
Let $X\subseteq\p5$ be a general cubic hypersurface of
$\hat \cC_{18}$. Then
\[\dim(\fI_X(2;1))=1.\]
\end{corollary}
\begin{proof}
The smooth Del Pezzo surfaces of degree $6$ in $\p6$ correspond to points in $\Hilb^{3t^2+3t+1}(\p6)$ by \cite[Exercise V.1.2]{Ha2}. Each such a Del Pezzo surface is abstractly isomorphic to the blow up $D$ of $\p2$ at the three fundamental points. Thus the aforementioned locus  is dominated by $\PGL_7$, hence irreducible. In order to give a smooth Del Pezzo surface $Y\subseteq\p5$ of degree $6$ we must project a smooth Del Pezzo surface of degree $6$ from a point in $\p6$ not lying on its secant variety, hence the locus ${\mathfrak D}_6\subseteq\Hilb^{3t^2+3t+1}(\p5)$ corresponding to such surfaces is dominated by an open  subscheme of $\PGL_7\times\p6$, hence it is irreducible as well.

Thanks to the Nakano vanishing theorem (see \cite[Corollary 6.4]{Es--Vie}) we know that $h^2(\cT_Y)=0$. Moreover, $Y$ is $\p2$ blown up at $3$ general points, hence $\Aut(Y)$ is the subgroup of the group of projectivities of $\p2$ fixing the blown up points: it follows that $h^0(\cT_Y)=\dim(\Aut(Y))=2$ (see \cite[Exercise I.2.16.4]{Kol}). The Riemann--Roch theorem on $Y$ then returns $h^1(\cT_Y)=0$. Finally the Kodaira vanishing theorem on $Y$, the cohomology of \eqref{seqNormal} with $Z:=Y$ and $P:=\p 5$ and the one of \eqref{seqTangentP5} restricted to $Y$ yield
$$
h^0(\cN_{Y\vert\p5})=39,\qquad h^1(\cN_{Y\vert\p5})=h^2(\cN_{Y\vert\p5})=0.
$$
Thus $\mathfrak D_6$ is smooth and its dimension is $39$. Let
$$
\mathfrak  J:=\{\ (Y,X)\ \vert\ Y\subseteq X\ \}\subseteq\mathfrak D_6\times \mathfrak X_3.
$$
There is an obvious projection $\mathfrak  d_6\colon\mathfrak J \to\mathfrak D_6$ whose fibre over $Y$ is the projective space of cubic hypersurfaces containing $Y$. Thanks to \cite[Proposition 4.1]{Ah--Kw} we know that
$$
h^0(\cI_{Y\vert\p5}(3))=h^0(\cI_{D\vert\p6}(3))-{8\choose6}.
$$
Since $\cI_{D\vert\p6}$ has a sheafified minimal free resolution of the form
$$
0\longrightarrow\cO_{\p6}(-6)\longrightarrow\cO_{\p6}(-4)^{\oplus9}\longrightarrow\cO_{\p6}(-3)^{\oplus16}\longrightarrow\cO_{\p6}(-2)^{\oplus9}\longrightarrow\cI_{D\vert\p6}\longrightarrow0
$$
we easily obtain $h^0(\cI_{Y\vert\p5}(3))=47-28=19$, hence the fibre of $\mathfrak d_6$ is $\p{18}$ and $\mathfrak D_6$ is smooth and irreducible of dimension $57$.

We also have a second projection $\mathfrak x\colon\mathfrak J\to \mathfrak X_3$ whose fiber over $X$ parameterizes smooth Del Pezzo surfaces $Y$ of degree $6$ inside $X$: in particular $\im(\mathfrak x)\subseteq\hat{\cC}_{18}$ by Proposition \ref{pInstanton1}. Such surfaces $Y$ are parameterized by pairs $(\cE,\sigma)$ where $\cE$ corresponds to a point in $\fI_X(2;1)$ and  $\sigma\in \bP(H^0(\cE))$. 

On the one hand, we found in Proposition \ref{pEquivalence15} that  $h^0(\cE)=3$, hence the fibre of $\mathfrak x$ over $X\in\im(\mathfrak x)$ has dimension $\dim(\fI_X(2;1))+2$. We deduce that
$$
54=\dim(\hat{\cC}_{18})\ge\dim(\im(\mathfrak x))=55-\dim(\fI_X(2;1))
$$
hence $\dim(\fI_X(2;1))\ge1$. Moreover, the general hypersurface corresponding to a point in $\cC_{18}$ contains a smooth Del Pezzo surface $Y$ of degree $6$ thanks to \cite[Proof of Theorem 6]{A--H--T--VA}, hence $\im(\mathfrak x)$ is dense inside $\hat{\cC}_{18}$. It follows that $\dim(\fI_X(2;1))=1$ generically.
\end{proof}

In the following statement we characterize smooth cubic hypersurfaces $X\subseteq\p5$ supporting a rank two instanton with $\quantum=1$ in terms of a class of surface that they contain.

\begin{proposition}
\label{pEquivalence15}
The fourfold $X$ supports a rank two instanton bundle $\cE$ with
charge $\quantum=1$ if and only if there is a smooth, connected
surface $Y\subseteq X$ not contained in any quadric and such that
\[\deg(Y)=15, \quad q(Y)=0, \quad p_g(Y)=6, \quad \omega_Y\cong\cO_Y(1).\]
\end{proposition}
\begin{proof}
If $\cE$ exists, then $\cE(1)$ is globally generated, hence the zero locus $Y$ of its general section is a smooth surface. The same argument used in the proof of Proposition \ref{pEquivalence14} implies that $Y$ has the desired properties.

Conversely, if $Y$ exists, then Theorem \ref{tSerre} implies the existence of a unique rank two vector bundle $\cA$ fitting into an exact sequence 
$$
0\longrightarrow\cO_X\longrightarrow\cA\longrightarrow\cI_{Y\vert X}(4h)\longrightarrow0.
$$
If $\cE:=\cA(-1)$, then $c_1(\cE)=2h$. Arguing as in the proof of Proposition \ref{pEquivalence14} we deduce that $\cE$ is a rank two instanton bundle on $X$. To compute its charge $k$, we can use \eqref{Chern}, that gives
 $c_2(\cE(1))h^2=\deg(Y)=15$, hence $\quantum=1$. 
\end{proof}

\subsubsection{Proof of Theorem \ref{18-theorem}} \label{18-proof}

The proof is scattered along the previous sections. Let us summarize
it here. We refere to the conditions \ref{18-i} to \ref{18-viii} of \S \ref{18-section}.
Note that \ref{18-i} obviously implies \ref{18-ii}. 

\bigskip

\paragraph{\ref{18-ii} \textit{is equivalent to} \ref{18-iii}} It is proved in
\cite{An--Cs2} that $X$ supports a rank-2 instanton of charge $k$ if
and only if $X$ is Steiner-Pfaffian of type type $\cF^\vee$, where,
according to \eqref{steiner}, $\cF$ is a vector bundle fitting into
\[
  0 \to \cO_{\bP^5}(-1) \xrightarrow{\vartheta} \cO_{\bP^5}^9 \to \cF \to 0.
\]

Since $\cF$ is locally free of rank 8, the presentation matrix $\vartheta$ must have
constant rank 1, so it must contain $6$ independent linear forms. The
remaining forms will depend linearly on these, so up to a suitable coordinate change, the transpose
of $\vartheta$ reads $(x_0,\ldots,x_5,0,0,0)$, so that $\cF \simeq \cO_{\bP^5}^3 \oplus
\cT_{\bP^5}(-1)$.
This shows \ref{18-ii} $\Leftrightarrow$ \ref{18-iii}.

\bigskip

\paragraph{\ref{18-iv} \textit{implies} \ref{18-iii}}
We borrow the setting from \cite{GSW, gs}, in the following sense.
In this paper, a Coble cubic $\Gamma$
is a hypersurface $\Gamma_\omega$ defined by an alternating 3-form $\omega$ over a 9-dimensional vector
space $V_9$. Note that in the literature a Coble cubic must be
singular precisely along an abelian surface. Incidentally, this
happens if and only if $\omega$
sits outside a discriminant locus which is best described in terms of
a complex reflection arrangement sitting inside a Cartan subspace of
$\wedge^3 V_9^\vee$, again we refer to \cite{gs}.
However, here we put no specific condition on $\omega$ besides the
fact that $\Gamma_\omega$ defined below is an irreducible cubic hypersurface.
To define $\Gamma_\omega$ we use the identification $\wedge^3 V_9^\vee \simeq H^0(\Omega_{\bP(V_9)}^2(3))$.
Then, $\omega \in \wedge^3 V_9^\vee$ gives a skew-symmetric morphism
\[\varphi_\omega  : \cT_{\bP(V_9)}(-2) \to \Omega_{\bP(V_9)}(1).\]
Assuming that the Pfaffian of $\varphi_\omega$ is non-zero (i.e., that
$\varphi_\omega$ is injective), we get a Steiner-Pfaffian hypersurface
of type $\Omega_{\bP(V_9)}(1)$, which is the Coble cubic
$\Gamma_\omega$.

Assuming condition \ref{18-iv}, there is an alternating 3-form
$\omega$ and a linear $\bP^5$ inside $\bP(V_9)$
such that $X = \Gamma_\omega \cap \bP^5$, so $X$ is the Pfaffian of
$\varphi_\omega|_{\bP^5}$. Since $\Omega_{\bP(V_9)}(1)$ restricts to
$\bP^5$ as $\cF^\vee = \cO_{\bP^5}^{3} \oplus \Omega_{\bP^5}(1)$, 
we get that $X$ is Steiner-Pfaffian of type $\cF^\vee$ and \ref{18-iii} is satisfied.

\bigskip

\paragraph{\ref{18-iii} \textit{implies} \ref{18-iv}}
Consider vector spaces $V_3$ and $V_6$ of dimension $3$ and $6$ and 
let us assume that $X \subset \bP(V_6)$ is Steiner-Pfaffian of type
$\cF^\vee = V_3^\vee \otimes \cO_{\bP(V_6)} \oplus \Omega_{\bP(V_6)}(1)$, so
that $X = \Pf(\varphi)$, for some skew-symmetric morphism $\cF(-1) \to
\cF^\vee$. Then we have
\begin{align*}
  \varphi \in H^0(\wedge^2 \cF^\vee(1)) &\simeq 
                                          H^0(\Omega_{\bP(V_6)}^2(3)) \oplus V_3^\vee \otimes   H^0(\Omega_{\bP(V_6)}(2))
                                          \oplus \wedge^2V_3^\vee \otimes   H^0(\cO_{\bP(V_6)}(1))  \simeq\\
                                        & \simeq \wedge^3 V_6^\vee \oplus V_3^\vee \otimes \wedge^2 V_6^\vee
                                          \oplus \wedge^2 V_3^\vee
                                          \otimes V_6^\vee.
\end{align*}
Set $V_9 = V_3 \oplus V_6$ and consider the decomposition:
\[
  \wedge^3 V_9^\vee = \wedge^3 V_6^\vee \oplus V_3^\vee \otimes \wedge^2 V_6^\vee
                                          \oplus \wedge^2 V_3^\vee
                                          \otimes V_6^\vee \oplus \wedge^3 V_3^\vee.
\]
Then, any choice of $\lambda \in \wedge^3 V_3^\vee$, together with
$\varphi$, determines an alternating 3-form $\omega(\lambda)$ over $V_9$ such
that the restriction to $\varphi_{\omega(\lambda)}$ tp $\bP(V_6)$ is
$\varphi$. The Pfaffian of $\varphi_{\omega(\lambda)}$ is thus non-zero and
defines a Coble cubic $\Gamma_{\omega(\lambda)} \subset \bP(V_9)$ such that $X =
\Gamma_{\omega(\lambda)} \cap \bP(V_6)$.

This proves along the way that, for any instanton $\cE$ of
charge 1 on $X$, there is a one-parameter
family of Coble cubics $(\Gamma_{\omega(\lambda)} \mid \lambda \in
\bC)$ such that $X = \Gamma_{\omega(\lambda)} \cap \bP(V_6)$ and  with the
additional property that for any value of the parameter $\lambda$, the instanton
$\cE= \coker(\varphi)$ is unchanged.

\bigskip

\paragraph{\ref{18-ii} \textit{is equivalent to} \ref{18-v}}

This is Proposition \ref{pInstanton1}.
Note that the dimension of the moduli
space of rank-2 instantons of charge 1 is computed by Corollary \ref{cCharge1}.

\bigskip

\paragraph{\ref{18-ii} \textit{is equivalent to} \ref{18-vi}}

This is Proposition \ref{pEquivalence15}.

\bigskip

\paragraph{\ref{18-i} \textit{implies} \ref{18-vii}}

This follows from Theorem \ref{theorem ulrich}.

\bigskip

\paragraph{\textit{Any condition implies} \ref{18-viii}}

We mentioned in \S \ref{section:instanton of rank 2} that \ref{18-ii}
implies \ref{18-viii}. Also, Proposition \ref{list of deltas} says
that \ref{18-vii} implies \ref{18-vii}. Note that $\dim(\fU_X(3))$
follows from the same proposition.

\bigskip

\paragraph{\textit{A generic cubic in  $\hat \cC_{18}$} verifies conditions
  \ref{18-i} to   \ref{18-viii}} Indeed, by Corollary \ref{cCharge1}
we know that if $X$ is general enough in $ \hat \cC_{18}$, then $X$ carries
a one-dimensional family of instantons of charge 1, so condition
\ref{18-ii} holds. Since in the proof of Corollary \ref{cCharge1} we
identified a projective bundle over the moduli space $\fI_X(2,1)$ with a fibre of a morphism
among integral schemes, we have that $\fI_X(2,1)$ is generically
smooth of dimension $1$, so for an instanton bundle $\cE$
corresponding to a general point of $\fI_X(2,1)$ we have
$\dim(\Ext^1_X(\cE,\cE))=1$.
By Riemann-Roch, this implies that $\cE$ is unobstructed, so that $X$
actually satisfies \ref{18-i} and all conditions up to \ref{18-viii}.

\begin{remark}
  Another proof of the fact that a general cubic $X$ in $\hat \cC_{18}$
  supports an unobstructed instanton of charge $1$ can be given by
  \texttt{Macaulay2}.
  Indeed, taking a random projection $Y$ of a smooth Del Pezzo surface of degree
  6 in $\bP^5$ and a random cubic fourfold $X$ containing $Y$, it is
  an immediate check with  \texttt{Macaulay2} to see that
  $h^1(\cN_{Y|X})=1$. Then, by Lemma \ref{normale-ext}, the associated
  instanton bundle $\cE$ satisfies
  $\dim(\Ext^1(\cE,\cE))=1$.
\end{remark}

\subsection{General cubic hypersurfaces in $\cC_{20}$.} \label{20:section}
In this section we deal with the case of rank two instanton bundles
with charge $k=2$. 
Here we assume that $\bk$ has characteristic zero.
We already know from Table \ref{table-k-delta} that every cubic
hypersurface $X\subseteq\p5$ supporting
a rank-2 instanton bundle of charge 2 corresponds to a point in
$\cC_{20}$. Such an instanton gives rise to a Steiner-Pfaffian
representation and, as we will observe in a minute, generically the
associated Steiner bundles is $\cF^\vee = \Omega_{\bP^5}(1)^{\oplus 2}$.
Also, we know that such an unobstructed instanton gives rise
to an Ulrich bundle of rank 4. The main result of this section aims at
reverting these assertions, by proving that for a general cubic of
$\cC_{20}$ all these conditions are fulfilled.

\begin{theorem} \label{20-theorem}
  Let $X$ be a general cubic fourfold in $\hat{\cC}_{20}$. Then:
  \begin{enumerate}[label=\roman*)]
    \item $X$ supports a two-dimensional family of rank-2 instanton
      bundles of charge 2,
    \item $X$ supports a six-dimensional symplectic family of stable
      Ulrich bundles of rank $4$,
    \item $X$ is Steiner-Pfaffian of type $\Omega_{\bP^5}(1)^{\oplus 2}$.
    \end{enumerate}
\end{theorem}

This proves Theorem \ref{C20-intro} from the introduction. Again, once the existence of Ulrich bundles of rank four is ensured, the assertion about their stability is clear if $X$ does not lie in $\hat \cC_{14}$, see the argument in \S \ref{18-section}.

\subsubsection{Instanton bundles of charge 2 and surfaces}

Let us see how the presence of rank-2 instanton bundles of charge 2 is
reflected by the fact that our cubic fourfold $X$ contains certain
surfaces.
The story here is a bit different than in the foregoing discussion, as Del Pezzo surfaces (this time of
degree 7) are only
associated with special instanton bundles, namely those that have a
non-zero global section, while canonical surfaces arise from generic
instantons of charge 2.

\begin{proposition} \label{pInstanton2}
  A smooth cubic fourfold $X$ supports a rank-2 instanton bundle $\cE$
  of  charge 2 having $h^0(\cE) \ne 0$ if and only if $X$ contains a connected non-degenerate LCI
  surface $Y$ with
  \[
    \omega_Y\simeq \cO_Y(-1), \quad h^1(\cO_Y)=0, \quad h^0(\cO_Y(1))=8.
  \]
\end{proposition}

\begin{proof}
The proof is exactly the same as that of Proposition
\ref{pInstanton1}, except for the fact that this time we have to
assume $h^0(\cE)\ne 0$ in the direct implication and to prove it in
the converse. However, this is obvious since  \eqref{seqIdeal} with
$d=3$ and $s=0$ gives an injection of $\cO_X$ into $\cE$ so $h^0(\cE)
\ne 0$.
\end{proof}

\begin{remark}
For a general isomorphic projection $Y\subseteq\p5$  of a Del Pezzo surface of degree $7$ in $\p7$ we have $h^0(\cI_{Y\vert\p5}(2))=0$ by \cite[Proposition 5.3]{KapG}, whence $h^0(\cE)=1$.
A smooth cubic fourfold containing $Y$ thus carries an instanton bundle $\cE$ of rank $2$ and charge $2$ with $h^0(\cE)=1$.

We do not know if there are smooth cubic fourfolds carrying instanton bundles with $h^0(\cE)\ge2$. 
We note that, if such a bundle exists, then the geometry of $Y$ is rather peculiar. Indeed, we have $h^0(\cI_{Y\vert\p5}(2))\ge1$, so $Y$ is contained in a quadric. If such a quadric is smooth, then as in Example \ref{eSerre} one deduces that $Y$ is the zero locus of a rank two instanton bundle with charge $2$, contradicting Corollary \ref{cBound}. 
So, there are quadrics through $Y$ but none of these quadrics is smooth. Actually, by stability of $\cE$ we see that a pencil of sections gives an injective map $\lambda : \cO_X^{\oplus 2} \to \cE$ which in turn gives rise to a quadric section $Z = \bV(\det(\lambda))$ of $X$. The threefold $Z$ is anticanonical, namely $\omega_Z \simeq \cO_Z(-1)$, and $\cC=\coker(\lambda)$ is a reflexive sheaf satisfying $\cC^\vee \otimes \omega_Z \simeq \cC(-3)$, by Grothendieck duality. Then $\cC^\vee \simeq \cC(-2)$. Hence, if $\cC$ was a line bundle on $Z$, we would have $\cC^{\otimes 2} \simeq \cO_Z(2)$, which is not possible since the Picard group of $Z$ has no torsion and $\cC$ is not isomorphic to $\cO_Z(1)$.
Then $Z$ is actually a non-factorial quadric section of $X$.
\end{remark}

It is not true that the general point of $\cC_{20}$ corresponds to a
smooth cubic hypersurface as above supporting an instanton bundle $\cE$ of rank $2$ and charge $2$ with $h^0(\cE) \ne 0$. Actually, we will see that a
general cubic in $\cC_{20}$ supports instanton bundles of rank $2$ and charge $2$
but that typically they have no non-zero global section.
Nevertheless, we are able to characterize the presence of instanton
bundles of charge 2 via canonical surfaces as in Propositions \ref{pEquivalence14} and  \ref{pEquivalence15}.

\begin{proposition}
\label{pEquivalence16}
A smooth cubic fourfold $X$ supports a rank two instanton bundle $\cE$
with charge $\quantum=2$ and $h^0(\cE)=0$, if and only if there is a
smooth, connected surface $Y\subseteq X$, contained in no cubic other than $X$,
such that
\[
  \deg(Y)=16,\quad q(Y)=0, \quad p_g(Y)=6,\quad \omega_Y\cong\cO_Y(1).
\]
\end{proposition}

\begin{versionb}
\begin{proposition}
\label{pEquivalence16}
The following assertions hold.
\begin{enumerate}
\item If $X$ supports a rank two instanton bundle $\cE$ with charge $\quantum=2$ and $h^0(\cE)=0$, then $h^1(\cE)=0$, $h^0(\cE(1))=18$ and the general section of $H^0(\cE(1))$ vanishes on a smooth, connected surface $Y\subseteq X$ not contained in any quadric and such that $\deg(Y)=16$, $q(Y)=0$, $p_g(Y)=6$, $\omega_Y\cong\cO_Y(1)$.
\item If there  is a smooth, connected surface $Y\subseteq X$ not contained in any quadric and such that $\deg(Y)=16$, $q(Y)=0$, $p_g(Y)=6$, $\omega_Y\cong\cO_Y(1)$, then $X$ supports a rank two instanton bundle $\cE$ with charge $\quantum=2$ and the general section of $H^0(\cE(1))$ vanishes on a surface with the same properties as $Y$.
\end{enumerate}
\end{proposition}
\end{versionb}

\begin{proof}
If $X$ supports an instanton $\cE$ as in the statement, since $h^0(\cE)=0$, it follows from \eqref{ChiE} with $t=0,1$ that
$h^1(\cE)=0$, incidentally we get $h^0(\cE(1))=18$. In particular, $\cE$ is $1$-regular in the sense of
Castelnuovo-Mumford, hence its general section vanishes on a smooth
surface $Y$.
From now on the whole proof of the statement runs along the same lines of the proofs of  Propositions \ref{pEquivalence14} and \ref{pEquivalence15}, with the caveat that the assumption that $X$ is the only cubic containing $Y$ implies $h^0(\cI_{Y|X}(2))=0$ and thus that $Y$ is contained in no quadric. 
\end{proof}

\begin{remark}
Notice that if $X$ supports a rank two instanton bundle $\cE$ with charge $\quantum=2$ and $h^0(\cE)\ne0$, then the general section of $H^0(\cE(1))$ can still vanish on a surface with all the properties listed in the statement of Proposition \ref{pEquivalence16} except, possibly, the smoothness.
\end{remark}

\subsubsection{Steiner-Pfaffian cubic fourfolds of type
  $\Omega_{\bP^5}(1)^{\oplus 2}$.}

If a cubic fourfold $X$ supports a rank-2 instanton bundle of charge
2, then, as we mentioned in the introduction, $X$ is Steiner-Pfaffian of type $\cF^\vee$, where $\cF$ fits
into
\[
  0 \to \cO_{\bP^{n+1}}(-1)^{2} \xrightarrow{\vartheta} \cO_{\bP^{n+1}}^{12} \to \cF  \to 0,
\]
For a generic choice of $\vartheta$, we have $\cF \simeq
\cT_{\bP^5}(-1)^2$, see Remark \ref{rBrambilla}.
We focus on the generic situation and we put
\[
  \cTF = \cT_{\bP^5}(-1)^{\oplus 2}.
\]
Let us look at the skew-symmetric maps $\cTF(-1) \to \cTF^\vee$.
We have 
\begin{equation}
\label{Decomposition}
\begin{aligned}
\wedge^2\cTF^\vee(1)&\cong \wedge^2 \Omega_{\p5}(3)\oplus \Omega_{\p5}\otimes
                      \Omega_{\p5}(3)\oplus \wedge^2\Omega_{\p5}(3)\\
&\cong \Omega^2_{\p5}(3)^{\oplus 3}\oplus S^2\Omega_{\p5}(3).
\end{aligned}
\end{equation}
\begin{remark}
Notice that $\Omega^2_{\p5}(3)$ is globally generated and
$h^0(\Omega^2_{\p5}(3))=20$ (e.g. see \cite[Bott
formula]{O--S--S}). Since $H^0(S^2\Omega_{\p5}(3))=0$ (see
\cite[Theorem 1.3.1]{Ete}), it follows that $\wedge^2\cTF^\vee(1)$
is not globally generated. 
Thus, we cannot use the Bertini theorem in order to prove that the Pfaffian locus $\Pf(\varphi)$ of a general section $\varphi\in H^0(\wedge^2\cTF^\vee(1))$  is a smooth hypersurface. 
\end{remark}

\begin{proposition}
\label{pDominant2}
The Pfaffian locus $X$ of a general $\varphi\in
H^0(\wedge^2\cTF^\vee(1))$ is a smooth cubic hypersurface in $\p5$
representing a point in $\cC_{20}$.
Moreover, $\cE:=\coker(\varphi)$ is a rank-2 stable instanton bundle of charge $\quantum=2$ with $h^0(\cE)=0$ and 
\begin{gather*}
  h^1(\cE\otimes\cE^\vee)= 2,\qquad
  h^2(\cE\otimes\cE^\vee)=1, \qquad
  h^p(\cE\otimes\cE^\vee)=0,  \qquad \mbox{for $p \in \{3,4\}$.}
\end{gather*}
\end{proposition}
\begin{proof}
Using \texttt{Macaulay2} (see \cite{Macaulay2}) we can find $\varphi\in H^0((\wedge^2\cTF)^\vee(1))$ such  that $\Pf(\varphi)$ is smooth.
In particular we have an injective morphism $\varphi\colon \cTF(-1)\to\cTF^\vee$. Since $\det(\cTF)=\cO_{\p5}(2)$, one easily checks that $\det(\varphi)$ has degree $6$, hence $\Pf(\varphi)$ is a cubic hypersurface. Since the locus of $\varphi\in H^0((\wedge^2\cTF)^\vee(1))$ such that $\Pf(\varphi)$ is smooth is open, it follows that the general $X:=\Pf(\varphi)$ is smooth for the general $\varphi$. 
Thus we obtain an exact sequence of the form
$$
0\longrightarrow\cTF(-1)\mapright\varphi\cTF^\vee\longrightarrow\cE\longrightarrow0.
$$
Its cohomology gives $h^0(\cE)=0$. Moreover, $\cE$ is a rank two
locally free sheaf by \cite[Corollary 1.3]{An--Cs2} with charge $\quantum=2$. 
Thanks to Propositions \ref{pEquivalence16} and \ref{pSurfaceY}, $X$
contains a smooth, connected surface $Y$ such that $h^2Y=16$ and
$Y^2=92$, hence $X$ represents a cubic of $\cC_{20}$.
Let us compute the dimension of
$H^i(\cE\otimes\cE^\vee)\cong\Ext_X^i(\cE,\cE)$ for $i\ge0$.
This is roughly the same story as in Lemma \ref{normale-ext}, however
our result here is a bit more precise so we carry out the proof in
detail.
First of all, recall that $\cE(1)|_Y \simeq \cN_{Y|X}$. 
Since
$\cE$ is $\mu$-stable by Lemma \ref{pStable}, it follows that it is
also simple,
i.e. $h^0(\cE\otimes\cE^\vee)=1$. 

The cohomology of \eqref{seqIdeal} for $d=3$ and $s=1$ tensored by
$\cE^\vee(-1)\cong\cE(-3)$ implies 
\begin{equation}
\label{EEversusYE}
h^i(\cE\otimes\cE^\vee)=h^i(\cI_{Y\vert X}\otimes\cE(1))
\end{equation}
for each $i$. The cohomology of \eqref{seqStandard} tensored by $\cE(1)$ returns
\begin{equation}
\label{YEversusIE}
\begin{gathered}
h^0(\cO_Y\otimes\cE(1))=17+h^1(\cI_{Y\vert X}\otimes\cE(1)),\\ h^1(\cO_Y\otimes\cE(1))=h^2(\cI_{Y\vert X}\otimes\cE(1)),\\
h^2(\cO_Y\otimes\cE(1))=h^3(\cI_{Y\vert X}\otimes\cE(1)).
\end{gathered}
\end{equation}
We first check $h^2(\cN_{Y\vert X})=0$. By \eqref{seqNormal}, we
obtain
\begin{equation}
\label{BoundA}
h^2(\cN_{Y\vert X})\le h^2(\cO_Y\otimes\cT_X).
\end{equation}
By \eqref{seqTangentHypersurface} we obtain
\begin{equation}
\label{BoundB}
h^2(\cO_Y\otimes\cT_X)\le h^2(\cO_Y\otimes\cT_{\p5}).
\end{equation}
The cohomology of \eqref{seqTangentP5} returns an exact sequence of the form
$$
H^2(\cO_Y)\mapright{u} H^2(\cO_Y(1))^{\oplus6}\longrightarrow H^2(\cO_Y\otimes\cT_{\p5})\longrightarrow0.
$$
Thanks to the Serre duality the first map above is the dual of a map $v\colon H^0(\cO_Y)^{\oplus6}\to H^0(\cO_Y(1))$ because $\omega_Y\cong\cO_Y(1)$. The latter map is the cohomology of the restriction to $Y$ of the Euler map $\cO_{\p5}^{\oplus6}\to\cO_{\p5}(1)$. The cohomology of the commutative diagram
\begin{equation*}
\begin{CD}
@.@.0@.0\\
@.@.@VVV@VVV\\
@.@.\cI_{Y\vert\p5}^{\oplus 6}@>>>\cI_{Y\vert\p5}(1)\\
@.@.@VVV@VVV\\
0@>>>\Omega_{\p5}(1)@>>>\cO_{\p5}^{\oplus 6}@>>> \cO_{\p5}(1)@>>>0\\
@.@.@VVV @VV\varrho^Y_1 V\\
@.@.\cO_Y^{\oplus 6}@>>> \cO_Y(1) \\
@.@.@VVV@VVV\\
@.@.0@.0
\end{CD}
\end{equation*}
finally yields that $v$ is an isomorphism if
$$
h^0(\cI_{Y\vert\p5})=h^1(\cI_{Y\vert\p5})=h^0(\cI_{Y\vert\p5}(1))=h^1(\cI_{Y\vert\p5}(1))=0.
$$
the two former vanishings are obvious. The third one follows from the inequality $h^0(\cI_{Y\vert\p5}(1))\le h^0(\cI_{Y\vert\p5}(2))$, because $Y$ is not contained in any quadric. Thus $\varrho^Y_1$ is injective. Since $p_g(Y)=h^0(\cO_{\p5}(1))$, the fourth vanishing above follows.
Thus the same is true for $u$, hence $h^2(\cO_Y\otimes\cT_{\p5})=0$. By combining such a vanishing with \eqref{BoundA} and \eqref{BoundB} we obtain the claimed vanishing $h^2(\cN_{Y\vert X})=0$.

\medskip

We note that $\chi(\cN_{Y\vert X})=h^0(\cN_{Y\vert
  X})-h^1(\cN_{Y\vert X})$. The cohomology of \eqref{seqNormal}
combined with the cohomologies of the restriction of \eqref{seqTangentP5} to $Y$ and of \eqref{seqChainNormal} 
return
\begin{equation}
\label{ChiN}
\chi(\cN_{Y\vert X})=6\chi(\cO_Y\otimes\cO_{\p5}(1))-\chi(\cO_Y)-\chi(\cT_Y)-\chi(\cO_Y\otimes\cO_{\p5}(3)).
\end{equation}
Since $\omega_Y\cong \cO_Y(1)$ and $q(Y)=0$ we know that $\chi(\cO_Y)=\chi(\cO_Y(1))=7$. The Riemann-Roch theorem on $Y$ gives $\chi(\cO_Y(3))=55$. The Noether formula for $Y$ returns $c_2(\cT_Y)=68$, whence $\chi(\cT_Y)=-38$ by the Riemann-Roch theorem on $Y$. By substitution in \eqref{ChiN} we obtain $\chi(\cN_{Y\vert X})=18$.

\medskip

Next, we use \texttt{Macaulay2} (see \cite{Macaulay2}) to obtain
$h^0(\cN_{Y\vert X})=19$ and $h^1(\cN_{Y\vert X})=1$. To do this, we
choose a random skew-symmetric map $\varphi : \cTF(-1) \to \cTF^\vee$,
which can be done by selecting 3 skew-symmetric maps
$\varphi_1,\varphi_2,\varphi_3, : \cT_{\bP^5}(-2)
\to \Omega_{\bP^5}(1)$ defined by 3 alternating 3-forms in 6 variables
and combinining them to get a skew-symmetric map
\[\varphi =
  \begin{pmatrix}
    \varphi_1 & \varphi_2 \\
    \varphi_2 & \varphi_3
  \end{pmatrix}.
\]
The Pfaffian of $\varphi$ defines a cubic fourfold $X$ and we may
define a random surface $Y$ arising as zero-locus of a section of
$\cE(1)$ by selecting a random submatrix of corank one from a
presentation matrix of $\cE$.
The computer-aided calculation gives $h^1(\cN_{Y\vert X})=1$ for this
surface, and we know that 1 is the minimal value of $h^1(\cN_{Y\vert
  X})$ as we deform $Y$ flatly, since the associated instanton $\cE$
then has $h^2(\cE^\vee \otimes \cE)=1$, which is the minimal possible
value according to Lemmas \ref{normale-ext} and \ref{at-least-1}, see \eqref{one-d}. Then, by semicontinuity,
$h^1(\cN_{Y\vert X})=1$ for a generic choice of a surface $Y \subset
X$ having the invariants fixed by our hypothesis. 
Taking into account of these computations and of \eqref{EEversusYE}
and \eqref{YEversusIE},
we finally obtain all the remaining values of $h^i(\cE\otimes\cE^\vee)$.
This finishes the proof.
\end{proof}

We have a natural rational map
$$
\Pf\colon H^0(\wedge^2\cTF^\vee(1))\dashrightarrow \hat{\cC}_{20}
$$ 
defined by $\varphi\mapsto\Pf(\varphi)$. If $X\in \im(\Pf)$, then we have also another map
$$
\Phi_X\colon\Pf^{-1}(X)\longrightarrow\fI_X(2;2)
$$
defined by $\varphi\mapsto\coker(\varphi)$. 

In what follows we will show that the fibres of $\Phi_X$ have all dimension at most $\dim(\GL_2)=4$. To this purpose, let $\varphi,\varphi_0\in \Pf^{-1}(X)$ and set $\cE:=\coker(\varphi)$, $\cE_0:=\coker(\varphi_0)$. Thanks to \cite[Proposition 4.2]{An--Cs2}, each isomorphism $w\colon \cE\to\cE_0$ lifts uniquely to an isomorphism of exact sequences
\begin{equation*}
\begin{CD}
0@>>>\cTF(-1)@>\varphi>>   \cTF^\vee@> q>> \cE@>>>0\ \ \\
@.@V v VV @V u VV @V w VV \\
0@>>>\cTF(-1)@>{\varphi_0}>>   \cTF^\vee@>q_0>> \cE_0@>>>0\ .
\end{CD}
\end{equation*}
In particular $v=u_{\vert \im(\varphi)}$. Thus $\Phi_X^{-1}(\coker(\varphi))$ is the orbit of the action of the group $\Aut(\cTF)\cong \GL_2$ on $H^0((\wedge^2\cTF)^\vee(1))$ given by $(u,\varphi)\mapsto u\varphi (u_{\vert \im(\varphi)})^{-1}$. 

Let $P$ be a component of $\Pf^{-1}(X)$ and denote by $\fI\subseteq\fI_X(2;2)$ be any component containing $\Phi_X(P)$. It follows that
\begin{equation}
\label{DimFibre}
\dim(P)\le \dim(\Phi_X(P))+4\le \dim(\fI)+4.
\end{equation}

The following corollary is an almost immediate byproduct of Proposition \ref{pDominant2}

\begin{corollary}
\label{cDominant}
The map $\Pf$ is dominant.
\end{corollary}
\begin{proof}
Since the fibers of $\mathfrak c$ are $\PGL_6$-orbits, hence
irreducible, it follows that $\hat{\cC}_{\delta}$ is irreducible
as well.
On the one hand, in Proposition \ref{pDominant2}, we showed the existence of a smooth cubic hypersurface $X\in\hat{\cC}_{20}$ endowed with an $h$-instanton bundle $\cE$ representing a point in $\im(\Phi_X)$ lying in a component of  $\fI_X(2;2)$ of dimension at most $2$, hence $\Pf^{-1}(X)$ has a component of dimension at most $6$ by \eqref{DimFibre}.
On the other hand, if  $\Pf$ is not dominant, then $\dim(\im(\Pf))\le53$, hence all the components of $\Pf^{-1}(X)$ should have dimension at least $7$ for each cubic hypersurface $X\in\hat{\cC}_{20}$  (e.g. see \cite[Exercise II.3.22 (b)]{Ha2}).
The contradiction yields that $\Pf$ is dominant.
\end{proof}

\begin{remark}
\label{rUnirational20}
In  \cite[Theorem 1.2]{Nuer} the author also shows that $\cC_{20}$ is unirational.
The unirationality of $\cC_{20}$ is also an immediate consequence of
Corollary \ref{cDominant}: indeed we prove therein that $\cC_{20}$ is
dominated by $H^0(\wedge^2\cTF^\vee(1))$.
\end{remark}

\begin{remark}
  \label{rBrambilla}
As we said, if $\cE$ is a rank two instanton on $X$ with charge $\quantum=2$, then there exists a vector bundle $\cF$ fitting into an exact sequence
\begin{equation}
\label{seqInstanton}
0\longrightarrow\cO_{\p5}(-1)^{\oplus2}\mapright\vartheta\cO_{\p5}^{\oplus12}\longrightarrow\cF\longrightarrow0
\end{equation}
and a skew-symmetric morphism $\varphi\colon\cF(-1)\to\cF^\vee$ such
that $\cE\cong\coker(\varphi)$. In the description above we restricted
ourselves only to $\cTF$ because \cite[Theorems 3.4 and 6.3]{Bra}
yield $\cF\cong\cTF$ for a general morphism $\vartheta$ and actually each time that $h^0(\cF^\vee)=0$.
This fact was also observed by Kac in a more general setting, \cite{kac:root}.
However, for special choices of $\vartheta$,
$\coker(\vartheta)$ may decompose as $\cO_{\bP^{n+1}}^{\ell} \oplus
\coker(\vartheta')$, where $\ell=h^0(\cF^\vee)$, with $1 \le \ell \le 5$ and
\[
  0 \to \cO_{\bP^{5}}(-1)^{2} \xrightarrow{\vartheta'} \cO_{\bP^{5}}^{12-\ell} \to \coker(\vartheta')  \to 0,
\]
Anyway, when $\cF\cong\cTF$, then $h^0(\cE)=0$, so the existence of
instanton bundles coming from Steiner bundles not isomorphic to $\cF_0$ is clear. Indeed, instanton bundles associated to Del Pezzo surfaces of degree $7$ as in Proposition \ref{pInstanton2} necessarily fit into \eqref{seqInstanton} with $\cF\not\cong\cTF$ because $h^0(\cE)\ne0$.
Also, as a consequence we get that the ideal of the surface $Y\subset \bP^5$ of Proposition \ref{pEquivalence16} has the self-dual resolution
\[
0 \to \cO_{\bP^5}(-4) \to \cT_{\bP^5}(-2)^{\oplus 2} \oplus \cO_{\bP^5}(-1)\to \Omega_{\bP^5}(1)^{\oplus 2} \oplus \cO_{\bP^5} \to \cI_{Y|\bP^5}(3) \to 0.
\]
\end{remark}

\section{Rank two Ulrich bundles on smooth quartic hypersurface} \label{section:quartic}

Some of the results about cubic fourfolds can also be formulated for
hypersurface of higher degree in $\bP^5$.
Let us quickly mention one result about instantons on quartic fourfolds.
This complements \cite{kim:ulrich-quartic}, where Ulrich bundles on quartic fourfolds are studied. However, we do not investigate here the potential connection between instanton and Ulrich bundles on these fourfolds.

\begin{lemma}
\label{pQuartic0}
A smooth quartic fourfold $X\subseteq\p5$ supports a rank two Ulrich bundle $\cE$ if and only if there is a smooth, connected surface $Y\subseteq X$ not contained in any quadric and such that $\deg(Y)=14$, $q(Y)=0$, $p_g(Y)=6$, $\omega_Y\cong\cO_Y(1)$.
\end{lemma}
\begin{proof}
If $\cE$ exists, then we can follow verbatim the proof of Proposition \ref{pEquivalence14}.
Conversely, assume the existence of a surface as in the statement. Theorem \ref{tSerre} implies the existence of an exact sequence 
$$
0\longrightarrow\cO_X\longrightarrow\cE\longrightarrow\cI_{Y\vert X}(3h)\longrightarrow0.
$$
From now on the proof is analogous to the one of Proposition \ref{pEquivalence14}.
\end{proof}

Lemma \ref{pQuartic0} and Remark \ref{rCatanese} yield that $X$ is Pfaffian if and only if it contains a smooth surface $Y$ with sheafified minimal free resolution
\begin{equation}
\label{seqSurface14}
0\longrightarrow\cO_{\p5}(-7)\longrightarrow\cO_{\p5}(-4)^{\oplus7}\longrightarrow\cO_{\p5}(-3)^{\oplus7}\longrightarrow\cI_{Y\vert\p5}\longrightarrow0.
\end{equation}

\begin{proposition}
  Let $X\subseteq\p5$ be a general linear Pfaffian quartic hypersurface. Then $\dim(\fI_X(2;0))=0$.
\end{proposition}
\begin{proof}
Consider the locus $\mathfrak X_4\subseteq\vert\cO_{\p5}(4)\vert$ of smooth quartic hypersurfaces and ${\mathfrak Y}\subseteq\Hilb^{7t^2-7t+7}(\p5)$ be the locus of surfaces $Y\subseteq\p5$ as in the statement. The locus ${\mathfrak Y}$ is irreducible because it is dominated by the space of $7\times7$ skew-symmetric matrices with entries in $H^0(\cO_{\p5}(1))$. We have the incidence relation 
$$
\mathfrak  J:=\{\ (Y,X)\ \vert\ Y\subseteq X\ \}\subseteq\mathfrak Y\times \mathfrak X_4.
$$
There are obvious projections $\mathfrak y\colon \mathfrak J\to\mathfrak Y$ and $\mathfrak x\colon \mathfrak J\to\mathfrak X_4$. Notice that $\im(\mathfrak x)$ is exactly the locus of linear Pfaffian quartic hypersurfaces thanks to Proposition \ref{pQuartic0}, hence its dimension is $104$ by Example \ref{ePfaffian}.

Thanks to \cite[Theorem 2.6]{Kl--MR} we know that $\dim({\mathfrak Y})=77$. Moreover, the cohomology of \eqref{seqSurface14} tensored by $\cO_{\p5}(4)$ returns $h^0(\cI_{Y\vert \p5}(4))=35$ when $Y\in \mathfrak Y$. Thus $\dim(\mathfrak J)=111$.

The points in the fibre of $\mathfrak x$ over $X\in\im(\mathfrak x)$ are parameterized by the pairs $(\cE,\sigma)$ where $\cE\in\fI_X(2;0)$ is a rank two Ulrich bundle on $X$ and $\sigma\in H^0(\cE)$ is a general section up to scalars. Since $h^0(\cE)=8$, we deduce that the fibre over $X$ has dimension 
$\dim(\fI(2;0))+7$. Thus 
$$
104=\dim(\im(\mathfrak x))=111-(7+\dim(\fI_X(2;0)))=104-\dim(\fI_X(2;0))
$$
where $X\in\mathfrak X_4$ is general, hence $\dim(\fI_X(2;0))=0$ for such an $X$.
\end{proof}

\begin{question}
In \cite[\S 8.10]{Bea} it is asked how many ways there are to write a generic quintic threefold as a linear Pfaffian, up to conjugacy of the generic linear group. This number is finite but there is no attempt in the literature at computing it explicitly.
In view of the above proposition, also a generic Pfaffian quartic fourfold can be written in a finite number of ways as a linear Pfaffian, up to conjugacy. It is natural to ask what this number is.
 \end{question}

\bibliographystyle{amsalpha}
\bibliography{bibliography}

\end{document}